\documentclass[preprint, 11pt]{amsart}
\usepackage{amsfonts,amssymb,amsmath,amsthm,amscd,mathtools,multicol,tikz, tikz-cd,caption,enumerate,mathrsfs,geometry,tgpagella}
\usetikzlibrary{arrows,chains,matrix,positioning,scopes}
\usepackage{inputenc}
\usepackage[foot]{amsaddr}
\usepackage[pagebackref=true, colorlinks, linkcolor=blue, citecolor=red]{hyperref}
\usepackage{fullpage}
\usepackage[all]{xy}

\makeindex

\newtheorem{theorem}{Theorem}[section]
\newtheorem{proposition}[theorem]{Proposition}
\newtheorem*{theorem*}{Theorem}
\newtheorem{corollary}[theorem]{Corollary}
\newtheorem{lemma}[theorem]{Lemma}
\newtheorem{remark}[theorem]{Remark}
\theoremstyle{remark}
\newtheorem{example}[theorem]{Example} 
 
\newtheorem{conjecture}[theorem]{Conjecture}
\numberwithin{equation}{subsection}

\newcommand{\bt}{\begin{theorem}}
\newcommand{\et}{\end{theorem}}
\newcommand{\bts}{\begin{theorem*}}
\newcommand{\ets}{\end{theorem*}}
\newcommand{\bco}{\begin{corollary}}
\newcommand{\eco}{\end{corollary}}
\newcommand{\bd}{\begin{definition}}
\newcommand{\ed}{\end{definition}}
\newcommand{\bp}{\begin{problem}}
\newcommand{\ep}{\end{problem}}
\newcommand{\bl}{\begin{lemma}}
\newcommand{\el}{\end{lemma}}
\newcommand{\bprop}{\begin{proposition}}
\newcommand{\eprop}{\end{proposition}}
\newcommand{\br}{\begin{remark}}
\newcommand{\er}{\end{remark}}
\newcommand{\bpf}{\begin{proof}}
\newcommand{\epf}{\end{proof}}
\newcommand{\bex}{\begin{example}}
\newcommand{\eex}{\end{example}}

\newcommand{\C}{\mathbb{C}}

\newcommand{\gal}{\mathscr{G}}
\newcommand{\galsub}{\mathscr{H}}

\renewcommand{\a}{\alpha}

\newcommand{\s}{\sigma}

\title{A classification of first order differential equations}
\author{
	Partha Kumbhakar, 
	Ursashi Roy,
	Varadharaj R. Srinivasan}
\address{Indian Institute of Science Education and Research Mohali, India.}
\email{\tiny{ parthakumbhakar@iisermohali.ac.in; ursashiroy@iisermohali.ac.in; ravisri@iisermohali.ac.in}}

\begin{document}
\maketitle
\begin{abstract} 
Let $k$ be a differential field of characteristic zero with  an algebraically closed field of constants. In this article, we provide a classification of first order differential equations over $k$ and study the  algebraic dependence of solutions of a given first order differential equation. Our results generalize parts of the work of Noordman et al. \cite{TvM22} and complements the work of Freitag et al. \cite{FJM22}.   \end{abstract}
\emergencystretch 3em
\section{Introduction}\label{intro} Let $k$ be a differential field of characteristic zero having an algebraically closed field of constants $C$. Let $f\in k[Y,Z]$ be an  irreducible polynomial involving the variable $Z.$ By a \emph{solution} of the differential equation $f(y,y')=0,$ we mean an element $u$ in a differential field extension  of $k$ having $C$ as its field of constants such that $f(u,u')=0.$ If $f$ is defined over the field of constants $C$ then $f(y,y')=0$ is called an \emph{autonomous first order differential equation}. The field of fractions $k(f)$ of the ring $k[Y,Z]/\langle f\rangle$ has a natural derivation induced by $f$  that maps $y$ to $z$ where $y,z$ are the images of $Y,Z$ in $k[Y,Z]/\langle f\rangle$ under the canonical projection.   It then follows that $f(y,y')=0$ has  a nonalgebraic solution if and only if the differential field $k(f)$ has $C$ as its field of constants.   

In this article, we study transcendence degree $1$ differential subfields of iterated strongly normal extensions. As a consequence of our study, we obtain a classification of first order differential equations $f(y,y')=0$ over $k$ into the following types:
\begin{enumerate} [(I)]
	\item\label{algebraictype} (\emph{algebraic type}) All solutions of $f(y,y')=0$ are algebraic over $k$ or equivalently, $k(f)$ with the induced derivation has a nonalgebraic constant.\\
	
	\item \label{ricattitype} (\emph{Riccati type}) There are a nonalgebraic solution of $f(y,y')=0,$ a finite algebraic extension $\tilde{k}$ of $k$ and an element $t\in \tilde{k}(f)$ such that $\tilde{k}(f)$ is a finite algebraic extension of $\tilde{k}(t)$ and that $t$ is a solution of a  Riccati differential equation: \begin{equation} \label{RDE}t'=a_2t^2+a_1t+a_0, \ \text{with}\ \ a_0, a_1, a_2\in \tilde{k}, \ \text{not all zero}.\end{equation}
	
	\item\label{weierstrasstype} (\emph{Weierstrass type}) There are a nonalgebraic solution of $f(y,y')=0,$ a finite algebraic extension $\tilde{k}$ of $k$ and an element $t\in \tilde{k}(f)$ such that $\tilde{k}(f)$ is a finite algebraic extension of $\tilde{k}(t,t')$ and that $t$ is a solution of a Weierstrass differential equation:  \begin{equation}  \label{WDE} (t')^2=\alpha^2(4t^3-g_2t-g_3), \  g_2, g_3 \in C, \alpha\in \tilde{k}\ \text{and}\   27g^2_3-g^3_2\neq 0.\end{equation} 
	
	\item\label{generaltype}  (\emph{general type}) \label{gentype} The differential equation $f(y,y')=0$  is not of the above types.  
	\end{enumerate}
 When $k=C,$ we obtain the  classification of autonomous differential equation due to Noordman et al.  \cite[page 1655]{TvM22}. If a differential equation is of Riccati or Weierstrass type then it is easy to see that $k(f)$ can be embedded in an iterated strongly normal extension. The following theorem contains the main results (Theorem \ref{trdonesubfields-SNE}, Theorem \ref{maintheorem} and  Corollary \ref{SNEoverC}) of this article.

\bt \label{intromaintheorem} Let $k$ be a differential field with an algebraically closed field of constants $C$ and $f(y,y')=0$ be a differential equation over $k.$ Suppose that $f(y,y')=0$ has a nonalgebraic solution  in an iterated strongly normal extension $E$ of $k$ or equivalently, the differential field $k(f)$ can be embedded in an iterated strongly normal extension $E$ of $k.$ Then the following statements hold. \begin{enumerate}[(i)] 
\item \label{RW-ISNE} The differential equation  $f(y,y')=0$ is of either Riccati or Weierstrass type. \\

\item \label{autonomoustype} If $k=C$ and $f(y,y')=0$ is of Riccati type  then there is an element $t\in C(f)$ such that
\begin{equation} \text{either}\  t'=1\ \ \text{or}\ \ t'=ct\ \ \ \text{for some nonzero}\  c\in C.  \end{equation}

\item \label{RW-SNE} If $E$ is a strongly normal extension of $k$ and  $f(y,y')=0$ is of Riccati type (respectively, Weierstrass type) then the finite algebraic extension $\tilde{k}$ and the element $t\in \tilde{k}(f),$ as in the definition of a Riccati type (respectively, Weierstrass type), can be chosen so that   $\tilde{k}(f)=\tilde{k}(t)$ (respectively, $\tilde{k}(f)=\tilde{k}(t,t')$). 
\end{enumerate}
\et

We use the above theorem to produce a family of geometrically irreducible plane algebraic curves $f$ over $k$ such that $f(y,y')=0$ is of general type (see Theorem \ref{generaltypegeneration}).  This family of examples of general type equations includes Abel differential equations (see Example \ref{Abel-DE}) of the form \begin{equation*}y'=a_ny^n+\cdots+a_2y^2,\end{equation*} 
 where $n\geq 3,$ $a_i\in k$ and both $a_2$ and $a_3$ have no antiderivatives in $k.$ Note that the differential equation $y'=y^3-y^2,$ that appears in \cite{Ros74} and \cite{TvM22}, belongs to the above family of Abel differential equations and therefore is of general type. We also prove the following results concerning algebraic dependence of nonalgebraic solutions (see Theorem \ref{algebraicdependence-generaltype}):

	\begin{enumerate} [(i)]
		
		\item If a first order differential equation over $k$ has a nonalgebraic solution in an iterated strongly normal extension of $k$ then any four nonalgebraic solutions from a differential field extension of $k$ having $C$ as its field of constants  are  $k-$algebraically dependent.\\
		
		\item If an autonomous differential equation has a nonalgebraic solution in an iterated strongly normal extension of $C$  then any two nonalgebraic solutions from a differential field  having $C$ as its field of constants are $C-$algebraically dependent.
		\end{enumerate}

In a recent article \cite{FJM22}, Freitag et al. study the algebraic dependence of solutions of differential equations of any order using model theoretic methods.  They show that if a first order differential equation (respectively, an autonomous differential equation) has four (respectively, two) algebraically independent solutions then any $m$ distinct nonalgebraic solutions are algebraically independent.  Thus, it is natural to make the following

\begin{conjecture} \label{conjecture}A first order differential equation over $k$ (respectively, An autonomous differential equation over $C$) is not of general type if and only if it has at most three (respectively, at most one) algebraically independent solutions in any given differential field extension of $k$ (respectively, $C$) having $C$ as its field of constants. 
	\end{conjecture}
We provide some evidence for this conjecture in Section \ref{rationalautoconjecture}. We will not prove any results concerning algebraic type differential equations. The articles  \cite{Mv07}, \cite{NNvT15}  and \cite{TvM22} are directly related to our work. We shall now explain in detail their connections to our main theorem.

\subsection{Classification of autonomous differential equations \cite{TvM22}.}  
Let $f\in C[Y,Z]$ be irreducible and $C(f)$ be the function field of the curve defined by $f.$   Let $\Gamma$ be the  projective nonsingular model whose function field $C(\Gamma)$ is isomorphic to $C(f).$  Since $C(f)$ is a differential field with the derivation induced by $f,$  we could also regard $C(\Gamma)$ as a differential field with the induced derivation $': C(\Gamma)\to C(\Gamma)$ .
	
\begin{minipage}{.75\textwidth} Let  $Der_C(C(\Gamma))$ be the $C(\Gamma)-$module of $C-$derivations on $C(\Gamma).$ Then, there is a natural $C(\Gamma)-$module isomorphism between $Der_C(C(\Gamma))$ and   $Hom_{C(\Gamma)}\left(\Omega^1_{ C(\Gamma)|C}, C(\Gamma)\right)$ 
given by $(a\mapsto a')\mapsto \omega,$ where $\omega$ is the differential such that $\Phi(\omega)=1.$ 
\end{minipage}
\begin{minipage}{.25\textwidth}
	\[
\begin{tikzcd} C(\Gamma) \arrow{r}{\mathrm d} \arrow{dr}[anchor=center, rotate=-32, yshift=1ex]{a\mapsto a'}&  \Omega^1_{C(\Gamma)|C}\arrow{d}{\Phi}\\  
	& C(\Gamma)
\end{tikzcd}
\]\end{minipage}

 Thus, to an autonomous differential equation, one can associate the tuple $(\Gamma, \omega).$
Using this identification, in \cite{TvM22}, autonomous equations are classified into four types:
\begin{enumerate}
	\item  \label{exacttype} $\omega=\mathrm dg$ for some $g\in C(\Gamma)$ or equivalently, there is an element $t	\in C(\Gamma)$ such that $t'=1;$ in this case the equations are of \emph{exact type}. \\
	\item \label{exponentialtyoe}  $\omega=\mathrm dg/cg$ for some $g\in C(\Gamma)$  and nonzero constant $c$ or equivalently, there is a $t\in C(\Gamma)$ such that $t'=ct$ for some nonzero constant $c;$ in this case the equations are of \emph{exponential type}. \\
\item \label{weierstrassiantype}$\omega= \mathrm dg/h,$ for $h,g\in C(\Gamma)$ and $h^2=g^3+ag+b$ with $4a^3+27b^2 \neq0$ or equivalently,  there is an element $t\in C(\Gamma)$ such that $t'^2 =t^3 +at+b$ for $a,b \in C$ with $4a^3+27b^2 \neq0$;  in this case the equations are of \emph{Weierstrass type}.\\

	\item \label{generaltype} If the equation is not of the above three types then it is of \emph{general type}. 
\end{enumerate}

Since $C$ is an algebraically closed field of constants, every nonconstant solution must be nonalgebraic
and every autonomous differential equation has a nonalgebraic solution (see Section \ref{rationalautoconjecture}). Thus, from Theorem \ref{intromaintheorem} (\ref{autonomoustype}), we obtain that our classification of differential equations coincides with the one above.

In \cite{TvM22}, an autonomous equation $(\Gamma, \omega)$ is called \emph{new} if $(\Gamma, \omega)$ is not a proper pull back.  Equations of type general and new are shown to have the  following  interesting property \cite[Theorem 2.1]{TvM22}: Any number of distinct nonalgebraic solutions  are $C-$algebraically independent.
Using this property, the authors prove that nonalgebraic solutions of equations of general type  cannot be found in any iterated Picard-Vessiot  extension \cite[Proposition 7.1]{TvM22}. In fact, their arguments can be extended to show that equations of general type do not have nonalgebraic solutions in any iterated strongly normal extension as well.   Thus, when $k=C,$ our theorem can also be recovered from their work. We would also like to point out that the \emph{type} of a rational autonomous equation; $y'=f(y),$ where $f$ is a nonzero rational function in one variable over $C,$ can be completely classified using \cite[Proposition 3.1]{Sri17}; for details see Remark \ref{rational-autoclass}.

\subsection{Painlev\'{e} property and differential equations of Ricatti or Weierstrass type \cite{Mv07}  \&  \cite{NNvT15}:} Let $k$ be a finite algebraic extension of the ordinary differential field $\C(x)$ of rational functions over complex numbers with $x'=1.$ A differential equation $f(y,y')=0$ over $k$ is said to have the \emph{Painlev\'{e} property} if the set of all branch points and the set of all essential singularities of the solutions form a discrete set.   A differential equation over $k$ of algebraic type has the Painlev\'{e} property if and only if all of its solutions lie in a fixed finite algebraic extension of $k$ (see \cite[Theorem 2.7 and Example 2.9]{NNvT15}).   

Let $f(y,y')=0$ have a nonalgebraic solution $u.$  Then the function field $k(f)$ with the induced derivation is isomorphic to $k(u,u')$ as differential fields. Since $k(u,u')$ has $\C$  as its field of constants,  from \cite[Theorem 4.5]{Mv07}, we deduce that $f$  has the Painlev\'{e} property  if and only if there exists a finite algebraic extension $\tilde{k}$ of $k$ and an element $t\in \tilde{k}(u,u')$ such that $\tilde{k}(u,u')=\tilde{k}(t,t')$ and that $t$ is a solution of either a Riccati  or a Weierstrass differential  equation over $\tilde{k}.$ In the former case, $\tilde{k}(t,t')$ can be embedded in a Picard-Vessiot extension of $\tilde{k}$ whose Galois group is isomorphic to a closed subgroup of $\mathrm{GL}_2(\C)$ and in the latter case, $\tilde{k}(t,t')$ is an abelian extension of $\tilde{k}$ whose Galois group is isomorphic to the $\C-$points of a nonsingular elliptic curve defined over $\C$ (see \ref{algebraicdependenceofsolutions} for details). 
  
Now,  from our Theorem \ref{intromaintheorem}(\ref{RW-SNE}), we can make the following observation: \begin{itemize} \item[] Let $k$ be a finite algebraic extension  of the ordinary differential field $\C(x)$ of rational functions in one variable $x$ with the derivation $d/dx.$ Let $f(y,y')=0$ be a first order differential equation with a nonalgebraic solution.  Then, $f(y,y')=0$ has a nonalgebraic solution in a strongly normal extension of a finite algebraic extension of $k$  if and only if $f(y,y')=0$ has the Painlev\'{e} property.	\end{itemize}

There are differential equations of Ricatti type  that do not have the Painlev\'{e} property. For example,  the differential equation $$y'=-\frac{1}{2x}y^3$$  has $u_c=1/\sqrt{\ln(x)+c}$ as a one-parameter family of solutions, where the parameter $c\in \C$.  Note that  $u_c\in \C(x)(\ln(x))\left(\sqrt{\ln(x)+c}\right)$ and that   $$ \C(x)\subset \C(x)(\ln(x))\subset \C(x)(\ln(x))\left(\sqrt{\ln(x)+c}\right).$$ Clearly, $\C(x)(\ln(x))(\sqrt{\ln(x)+c})$  is an  iterated (Picard-Vessiot) strongly normal extension of $\C(x).$ Since $\ln(x)\in \C(x)(u_c)$ and $(\ln(x))'=1/x\in \C(x),$ we see that $y'=-\frac{1}{2x}y^3$ is of Ricatti type.  However, the set of all branched points of these solutions, namely the set $\C,$ is not  discrete.

\subsection*{Organisation of the paper.}  In Section \ref{DGT}, we gather facts from the theory of strongly normal extensions for easy reference. In Section \ref{OPMT}, we provide an outline of the proof of our main theorem and in Section \ref{liouvillian-Picard-Vessiot}, using a work of Andr\'{e} Yves on solution algebras, we prove a structure theorem for differential subfields of liouvillian extensions. The proof of Theorem \ref{intromaintheorem} can be found in  Sections \ref{trdone-SN} and \ref{trdone-ISNE}. In Section \ref{application}, we  prove results concerning the algebraic dependence of solutions of first order differential equations. We also show that Conjecture \ref{conjecture} holds for rational autonomous differential equations. 
\section{Differential Galois Theory}   \label{DGT}

Here, we shall set up notations and put together the necessary definitions and results from the theory of strongly normal extensions. 

\subsection{Strongly normal extensions \cite{Kol53}, \cite{Kol55}, \cite{Kov03}, \cite{Kov06}.} \label{TSNE}  A differential field extension $E$ of $k$ is called a \emph{strongly normal extension} if $E$ is  finitely differentially generated over $k$  and if every differential $k-$isomorphism $\sigma$ of $E$  satisfies the following conditions: $E\sigma E=EC(\sigma) = \sigma E C(\sigma)$, where $C(\sigma)$ is the field of constants of the compositum $E\sigma E$ and that $\s(c)=c$ for all constants $c$ of $E$.   It can be shown that every strongly normal extension of  $k$ is a finitely generated field extension of $k$ and has $C$ as its field of constants (\cite[Propositions 12.2, 12.4]{Kov03}). A differential field extension $E$ of $k$ is called an \emph {iterated strongly normal extension} if there is a tower of differential fields $$k=E_0\subseteq E_1\subseteq E_2\subseteq \cdots\subseteq E_n=E$$ such that $E_i$ is a strongly normal extension of $E_{i-1}$ for each $i=1,\dots, n.$ When the field of constants is algebraically closed, strongly normal extensions exist and are unique up to differential isomorphisms. 

The group of all $k-$differential automorphisms of  a differential field extension $E$ of $k$ is called the \emph{differential Galois group} of $E$ over $k$ and is denoted by $\mathscr{G}(E|k)$ or simply by $\gal,$ when there is no ambiguity. If the field of constants of $E$ is $C$ and for any $x\in E\setminus k$ there is $\s\in \gal(E|k)$ such that $\s(x)\neq x$ then $E$ is called a \emph{weakly normal extension} of $k.$

For a strongly normal extension $E$ over $k,$ the elements of $\mathscr{G}(E|k)$ can be canonically identified with the $C-$points of a reduced group scheme of finite type defined over $C$ \cite[Sections 33-36 ]{Kov03}.  Thus, $\mathscr G (E|k)$ can be thought of as an algebraic group\footnote{Since $C$ is of characteristic zero and algebraically closed, group schemes of finite type over $C$ are  smooth.} (though not necessarily affine) defined over the field of constants $C$. The fundamental theorem of strongly normal extension (\cite[Theorem 36.3]{Kov03}) provides a bijective correspondence between the intermediate differential subfields
and the Zariski closed subgroups (closed subschemes) of $\mathscr{G}(E|k)$. If $\mathscr{H}$ is a closed subgroup of $\mathscr{G}(E|k)$ and $K$ is an intermediate differential field then the bijective correspondence is given by the maps
\begin{align*}
	&K \to \mathscr{G}(E|K):=\left\lbrace \sigma \in \mathscr{G}(E|k)\ |\ \sigma(u)=u \  \ \text{for all}\  \ u  \in K \right\rbrace \\
	& \mathscr{H} \to E^{\mathscr{H}}:= \left\lbrace u  \ \in E\ |\ \sigma(u)=u \ \ \text{for all} \ \sigma  \in \mathscr{H} \right\rbrace. 
\end{align*}
In this correspondence, the field fixed by the group $\mathscr{G}(E|k)$ is the base field $k$, that is $E^{\mathscr{G}(E|k)}=k$. Let $K$ be a differential field intermediate to $E$ and $k$. Then $K$ is a strongly normal extension of $k$ if and only if $\mathscr{G}(E|K)$ is a closed normal subgroup of $\mathscr{G}(E|k)$;
in which case, the differential Galois group $\mathscr{G}(K|k)$ is isomorphic to the quotient group $\mathscr{G}(E|k)/\mathscr{G}(E|K)$. 
The algebraic closure of $k$ in $E$ is a finite Galois extension, which we denote by $E^0$. 

A strongly normal extension $E$ over $k$ has the following structure (\cite[Proposition 12.2]{Kov06}), which  will  play a crucial role in our proofs: \begin{equation} k\subseteq E^0\subseteq F\subseteq E, \label{SNE-resolution} \end{equation} 
where $E^0$ is the relative algebraic closure of $k$ in $E$ and that $E^0$ is also a finite Galois extension of $k,$ $F$ is an \emph{abelian extension} of $E^0,$ that is, $F$ is strongly normal over $E^0$  such that $\gal(F|E^0)$ is isomorphic to an abelian variety defined over $C,$ and $E$ is a strongly normal extension of $F$ such that $\gal(E|F)$ is a connected linear algebraic group defined over $C.$ Such an extension $E$ is  a Picard-Vessiot extension of $F$ (see \cite[Theorem 2]{Kol55} or \cite[Sections 8-11]{Kov06}).

\subsection{Picard-Vessiot theory  \cite{Put-Sin}} \label{PVTheory}   A \emph{Picard-Vessiot ring} $R$ for a matrix differential equation $Y'=AY,$ where $A\in M_n(k),$ is a differential ring containing $k$ and having the following properties:

\begin{enumerate}[(a)]
	\item $R$ is a simple differential ring, that is, $R$ has no nontrivial differential ideals.\\
	 \item There is a \emph{fundamental matrix} $F$ for $Y'=AY$ with entries in $R,$ that is, a matrix $F\in \mathrm{GL}_n(R)$ such that $F'=AF.$\\  \item $R$ is minimal with respect to these properties, that is,  $R$ is generated as a ring by the elements of $k,$ the entries of $F$ and  inverse of the determinant of $F.$   \end{enumerate}

By a \emph{differential $k-$module} $(M,\partial),$ we mean a finite dimensional $k-$module $M$ and an additive map $\partial: M\to M$ such that $\partial(\alpha m)=\alpha'm+\alpha\partial(m)$ for all $\alpha\in k$ and $m\in M.$  Let $M$ be a differential $k-$module. By fixing a $k-$basis $e_1,\dots, e_n$ of $M,$ we obtain a matrix $A=(a_{ij})\in M_n(k)$ such that $\partial(e_i)=-\sum_j a_{ji}e_j$ and a  corresponding matrix differential equation $Y'=AY.$  Choosing any other basis will amount to obtaining a differential equation of the form $Y'=\tilde{A}Y,$ where $\tilde{A}=B'B^{-1}+BAB^{-1}$ for some $B\in \mathrm{GL}_n(k).$ Furthermore, if $R$ is a Picard-Vessiot ring for $Y'=AY$ with fundamental matrix $F$ then $(BF)'=\tilde{A}BF$ and thus $R$ is also the Picard-Vessiot ring for  $Y'=\tilde{A}Y.$ This observation allows one to define a \emph{Picard-Vessiot ring} for a differential module $M$ to be a Picard-Vessiot ring for a corresponding matrix differential equation $Y'=AY$ of $M.$ Picard-Vessiot rings are integral domains and their field of fractions are called  \emph{Picard-Vessiot extensions}. Let  $E$ be  a Picard-Vessiot extension of $k$ with Picard-Vessiot ring $R$ then from  \cite[Corollary 1.38]{Put-Sin} it follows that  $R$ is  the $k-$algebra generated by all the solutions in $E$ of all linear differential equations over $k.$ 

 Let $M$ be a $k-$differential module with matrix differential equation $Y'=AY$ and $M^{\vee}$ be the dual of a differential module of $M.$ Then $Y'=-A^t Y,$ where $A^t$ is the transpose of $A,$ is a matrix differential equation corresponding to $M^{\vee}.$ Thus, if $R$ is a Picard-Vessiot ring for $M$ with fundamental matrix $F\in \mathrm{GL}_n(R)$ then $\left((F^t)^{-1}\right)'=-A^t (F^t)^{-1}$ and thus $M$ and $M^{\vee}$ have the same Picard-Vessiot ring $R$.  Let $k[\partial]$ be the ring of differential operators over $k$ and $\mathscr L:=\partial^n+a_{n-1}\partial^{n-1}+\cdots+a_0\in k[\partial].$ Then for $M=k[\partial]/k[\partial]\mathscr L,$ a matrix equation corresponding to the dual $M^{\vee}$ is $Y'=A_\mathscr L Y,$  where  $A_\mathscr{L}$ is the companion matrix of $\mathscr L.$ Therefore,  if $R$ is a Picard-Vessiot ring for $M^{\vee}$ then the fundamental matrix $F\in\mathrm{GL}_n(R)$ for $Y'=A_\mathscr L Y$ is a Wronskian matrix.
$$A_\mathscr{L}=\begin{pmatrix}0&1&0&0&\cdots&0\\ 0&0&1&0&\cdots&0\\ \vdots&\vdots&\vdots&\vdots&\vdots&\vdots\\ 0&0&0&\cdots&\cdots&1\\ -a_0&-a_1&\cdots&\cdots&\cdots&-a_{n-1} \end{pmatrix},\qquad F=\begin{pmatrix}y_1&y_2&\cdots&\cdots&y_n\\ y'_1&y'_2&\cdots&\cdots&y'_n\\ \vdots&\vdots&\vdots&\vdots&\vdots\\ y^{(n-1)}_1&y^{(n-1)}_2&\cdots&\cdots&y^{(n-1)}_n \end{pmatrix}.$$
Note that $y_1,\dots,y_n$ are $C-$linearly independent and the $C-$vector space $V$ spanned by $y_1,\dots, y_n$ is the set of all solutions of $\mathscr L(y)=0.$

\section{Outline of the proof of the main theorem}\label{OPMT}

Here we shall discuss the strategies involved in the proofs of Theorem \ref{intromaintheorem} (\ref{RW-ISNE}) and (\ref{RW-SNE}). Let $f(y,y')=0$ be a differential equation having a nonalgebraic solution $u$ in a strongly normal extension $E$ of $k$ and let $K:=k(u,u').$ Observe that $\mathrm{tr.deg}(K|k)=1.$ 

The proof of Theorem \ref{intromaintheorem}(\ref{RW-SNE}) appears as  Theorem \ref{trdonesubfields-SNE}. It involves decomposing $E$ as in (\ref{SNE-resolution}) and considering the following cases:

\begin{minipage}{.8\textwidth}Case(i). Suppose that $\mathrm{tr.deg}(KF|F)=1.$ Since the  differential Galois group $\gal(E|F)$ is a  connected linear algebraic group,  $\gal(E|F)$ is a union of its Borel subgroups. Now $F$ is algebraically closed in $E$ implies that one can  choose  a Borel subgroup $\mathscr B$  such that $\mathrm{tr.deg}(KE^{\mathscr B}|E^{\mathscr B})=1.$  Since $\mathscr B$ is  (connected) solvable, $E$ is a liouvillian Picard-Vessiot extension of $E^{\mathscr B}.$ We classify differential subfields of liouvillian Picard-Vessiot extensions (Theorem \ref{LPV-structuretheorem}) and show that either $KE^{\mathscr U}=E^{\mathscr U}(t),$ where $\mathscr U$ is the unipotent radical of $\mathscr B$ and $t'\in E^{\mathscr U}$ or $KE^{\mathscr B}=E^{\mathscr B}(t),$ where $t'/t\in E^{\mathscr B}.$  We then use the fact that the geometric genus is an invariant under base change by separable fields  and obtain that $K$ is a genus zero extension of the algebraic closure of $k$ in $K$ (see \cite[Theorem 5]{Che51}).  Thus, one can find a finite algebraic extension $\tilde{k}$ of $k$ so that $\tilde{k}K=\tilde{k}(y).$ We then use the fact that $y=(at+b)/(ct+d)$ for $a,b,c,d \in \tilde{k}E^{\mathscr U}$ and show that $y$ satisfies a Riccati  equation over $\tilde{k}.$  
\end{minipage}
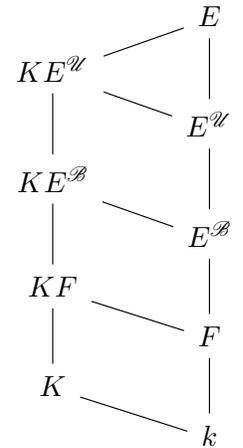
\begin{minipage}{.8\textwidth}
\hspace{1em}\begin{tikzcd} [row sep=.5ex]
	& E\ar[dash]{dl}\ar[dash]{dd}\\
	KE^{\mathscr U}\ar[dash]{dr}\ar[dash]{dd}&\\
	& E^{\mathscr U} \ar[dash]{dd}\\
	KE^\mathscr{B}\ar[dash]{dr}\ar[dash]{dd}&\\
	&E^\mathscr{B}\ar[dash]{dd}\\
	KF\ar[dash]{dr}\ar[dash]{dd}&\\
	&F\ar[dash]{dd}\\
	K\ar[dash]{dr}&\\
	& k\\
\end{tikzcd}
\end{minipage}

Case(ii). Suppose that $\mathrm{tr.deg}(KF|F)=0.$ Then since $F$ is algebraically closed in $E,$ we have $K\subseteq F.$ Now $KE^0$ is an intermediate differential field of a strongly normal extension, whose Galois group is abelian. Therefore $KE^0$ itself is a strongly normal extension of $E^0$ such that $\mathrm{tr.deg}(KE^0|E^0)=1.$ Then from \cite[Theorem 3]{Kol55}, we obtain a finite algebraic extension $\tilde{k}$ of $E^0$ such that $\tilde{k}K=\tilde{k}(t,t'),$ where $t$ is a solution of a  Weierstrass differential equation over $\tilde{k}.$

\begin{remark} \normalfont  In the event that $E$ is a Picard-Vessiot extension of $k,$  one can show that $KE^0$ is a rational field generated by a solution of a Riccati  differential equation.  To see this, we consider the connected group $\gal(E|E^0)$ and its codimension one closed subgroup $\gal(E|KE^0).$ Then, the function field of the homogeneous space $\gal(E|E^0)/\gal(E|KE^0)$ is known to be rational; for example, see \cite[Theorem 4.4]{CZ17}. Since $KE^0$ is isomorphic (as fields) to $E^0(\gal(E|E^0)/\gal(E|KE^0))=E^0(x),$ we obtain that $KE^0=E^0(y)$ \cite[p.87]{Mag94}.  Now, we have $KE^{\mathscr U}=E^{\mathscr U}(t)=E^\mathscr U(y),$ where $t'\in E^{\mathscr U}$ or $KE^{\mathscr B}=E^{\mathscr B}(t)=E^\mathscr B(y),$ where $t'/t\in E^{\mathscr B}.$  One then goes on to show that such a $y$ must satisfy a Riccati  differential equation over $E^0.$
	\end{remark}

Now we shall detail the strategies involved in  proving  Theorem\ref{intromaintheorem}(\ref{RW-ISNE}) (Lemma \ref{descent-lemma} and Theorem \ref{maintheorem}).   

Step 1. Let $E$ be a strongly normal extension of $k$ and $\tilde{E}$ be a finite algebraic extension of $E.$ Let $K$  be a differential field such that $k\subset K\subseteq \tilde{E}$ and that $\mathrm{tr.deg}(K|k)=1.$ Then using the facts that $KE$ is  strongly normal over $K$ and  that the differential Galois group $\gal(KE|K)$ injects as a closed subgroup of the differential Galois group $\gal(E|k)$ and that $KE$ is an algebraic extension of $E,$ we obtain that $\mathrm{tr.deg}((K\cap E)|k)=1.$ Now from the first part of the theorem, we conclude that there is a finite algebraic extension $\tilde{k}$ of $k$ and an element $y\in \tilde{k}K$ transcendental over $\tilde{k}$  such that $\tilde{k}(K\cap E)=\tilde{k}(y,y')$ and that $y$ is a solution of either a Riccati  differential equation or a Weierstrass differential equation over $\tilde{k}.$

Step 2. (A differential descent lemma)  Let $E$ be an iterated strongly normal extension of $k,$ $\overline{E}$ be an algebraic closure of $E$ and $K$ be a finitely generated differential field intermediate to $k$ and $\overline{E}$ with $\mathrm{tr.deg}(K|k)=1.$  

\begin{minipage}{.64\textwidth} Then one can assume that $$k=E_0\subseteq E_1\subseteq \cdots\subseteq E_{n-1}\subseteq E\subseteq \overline{E},$$ where $E_i$ is a strongly normal extension of $E_{i-1},$ $E$ is transcendental over $E_{n-1},$ and $\mathrm{tr.deg}(KE_{n-1}| E_{n-1})=1.$
Now from Step 1, we have $$\overline{E}_{n-1}\subseteq \overline{E}_{n-1}(y_{n-1},y'_{n-1})\subseteq \overline{E}_{n-1}K,$$where $y_{n-1}$ is a nonalgebraic solution of a Riccati  differential equation  or a Weierstrass differential equation over $\overline{E}_{n-1}$. We first show that $\overline{E}_{n-1}K$ is a weakly normal extension of $\overline{E}_{n-2}K$ and that $y_{n-1}$ can be chosen so that $\overline{E}_{n-1}(y_{n-1},y'_{n-1})$ is stabilized by the group $\gal(\overline{E}_{n-1}K|\overline{E}_{n-2}K).$ Using these facts, we then show that one can obtain a nonalgebraic solution $y_{n-2}\in \overline{E}_{n-2}K$ of some Riccati  differential equation or a Weierstrass differential equation over $\overline{E}_{n-2}.$ A repeated application of this ``descent argument"  proves the second part of the main theorem.  \end{minipage}
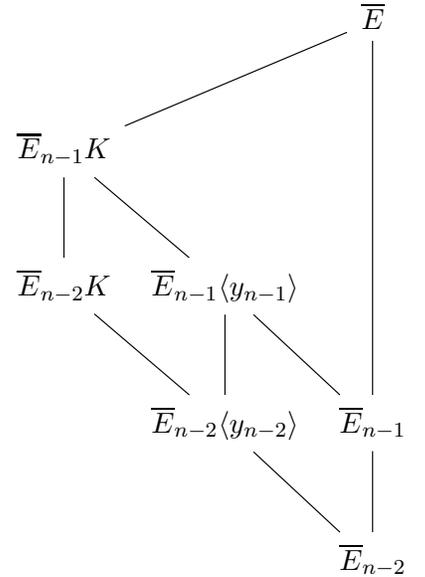
\begin{minipage}{.3\textwidth}
\hspace{1.8em}\begin{tikzcd}[row sep=6ex, column sep= 1ex]
	& &\overline{E}\ar[dash]{ddd}\ar[dash]{dll}\\
	\overline{E}_{n-1}K\ar[dash]{dr}\ar[dash]{d}& &\\
	\overline{E}_{n-2}K\ar[dash]{dr} & \overline{E}_{n-1}\langle y_{n-1}\rangle \ar[dash]{dr}\ar[dash]{d}& \\
	& \overline{E}_{n-2}\langle y_{n-2}\rangle\ar[dash]{dr}& \overline{E}_{n-1}\ar[dash]{d}\\
	&& \overline{E}_{n-2}
\end{tikzcd}
	\end{minipage}


	\section{Differential Subfields of Liouvillian Picard-Vessiot Extensions} \label{liouvillian-Picard-Vessiot}
	
  In  \cite{And14},   a Galois correspondence between certain differential subfields, called the solution fields, and the observable subgroups of the differential Galois group has been established. We shall now explain relevant definitions and results from \cite{And14} that  are needed to prove a structure theorem for the intermediate differential fields of liouvillian Picard-Vessiot extensions. 
	
	\subsection*{\bf Solution algebras and solution fields}	Let $k[\partial]$ be the usual ring of differential operators. Let $M$ be a differential $k-$module of dimension $n$.   A differential field extension $E$ of $k$ having the same field of constants as $k$ is called a \emph{solution field} for $M$ if there is a morphism $\psi: M\to E$ of $k[\partial]-$modules such that $\psi(M)$ generates $E;$ in which case,  $E$ is  said to be generated by a solution $\psi.$ Let $S$ be a differential $k-$algebra and an integral domain such that its field of fractions has field of constants $C$. Then $S$ is called a \emph{solution algebra} for $M$ if there is a  morphism $\psi: M\to S$ of $k[\partial]-$modules such that $\psi(M)$ generates $S.$ Let $E$ be a differential field extension of $k$ having the same field of constants as $k.$ For $i=1,\dots, n$, let $y_i\in E$ and $\mathscr{L}_i\in k[\partial]$ such that $\mathscr L_i(y_i)=0.$ Let $R$ be the differential $k-$subalgebra of $E$ generated by $y_1,\dots, y_n.$  Then $R$ is a solution algebra. To see this, consider the differential modules $k[\partial]/k[\partial]\mathscr L_i,$ where $\mathscr{L}_i(y_i)=0.$ Then the images of the  $k[\partial]-$morphisms $\oplus^n_{i=1}\psi_i: \bigoplus^n_{i=1}k[\partial]/k[\partial]\mathscr L_i\to E$ defined by $\psi_i(1)=y_i$ generates $R.$ Similarly, the differential field generated by $y_1,\dots, y_n$ is a solution field.  Conversely, let $S$ be a solution algebra generated by a solution $\psi.$  Then we have a map $\psi: M\to S$ of  $k[\partial]-$modules such that $\psi(M)$ generates $S$ as a $k-$algebra. Note that $\psi(M)$ is a (finite dimensional) differential $k-$module.  Let $y_1,\dots,y_n$ be a $k-$basis of $\psi(M).$ Then each $y_i$ must satisfy a linear homogeneous differential equation over $k.$ We shall summarize this observation in the following

	\bprop  \label{solutionfields-algebras}Solution fields (respectively, Solution algebras) over $k$ are generated as a field (respectively, as a $k-$algebra) by solutions of linear homogeneous differential equations over $k.$
	\eprop
	
	Let  $M^{\vee}$ be the dual of $M.$  A differential field  $E$  is called a \emph{Picard-Vessiot field} if it has the same field of constants as $k$ and the $C-$modules Sol$(M,E):=$ Hom$_{k[\partial]}(M,E)$  and Sol$(M^{\vee},E):=$ Hom$_{k[\partial]}(M^{\vee},E)$ are of dimension $n$ over $C$ and $E$ is minimal with respect to these properties.   As noted in Section \ref{DGT}, $M$ and $M^\vee$ have the same Picard-Vessiot extension and thus it is readily seen that the notions of a Picard-Vessiot field for $M$ and a Picard-Vessiot extension for $M$ are one and the same.  Observe that if $E$ is a Picard-Vessiot field of a differential $k-$module $M$ of dimension $n$ then $E$ is a solution field of  $M^n.$ 
\begin{center}
	
\end{center}

	\begin{theorem}(\cite[Lemma 4.2.2, Theorem 1.2.2]{And14})\label{observable} Let $\langle M\rangle^\otimes$ be the tannakian category over $C$ generated by the differential $k-$module $M.$
		\begin{enumerate}[(i)]
			\item The quotient field of a solution algebra $S$ for  $M$ is a solution field for $M$.\\
			\item Conversely, any solution field $K$ for $M$ is the quotient field of (non unique) solution algebra $\mathcal{S}$ for $M$.\\
			\item Any solution field for $N\in \langle M\rangle^\otimes$ embeds as differential subfield of a Picard-Vessiot field for $M.$\\
			\item Let $E$ be the Picard-Vessiot field for $M$ with Galois group $\gal$, then an intermediate differential field $k \subseteq K  \subseteq E$ is a solution field for $N\in \langle M\rangle^\otimes$ if and only if the corresponding subgroup $\galsub \subseteq \gal$ is observable (i.e. $\gal/\galsub$ is quasi-affine).
		\end{enumerate}
		
	\end{theorem}

  In view of the above theorem, it is natural to ask for a characterization theorem of those Picard-Vessiot extensions having only solution fields as intermediate differential fields.  Equivalently, characterize those linear algebraic groups whose closed subgroups are observable. We shall provide such a characterization theorem in the next proposition. But first, a few definitions are in order.
  
A differential field extension $E$ of $k$ is called a \emph{liouvillian extension} of $k$ if there is a tower of differential fields $$k=E_0\subseteq E_1\subseteq \cdots\subseteq E_n=E$$ such that $E_i=E_{i-1}(t_i),$ where $t_i$ is either algebraic over $E_{i-1}$ or $t'_i\in E_{i-1}$ or $t'_i/t_i\in E_{i-1}.$  If $E$ is both a Picard-Vessiot extension and a liouvillian extension of $k$ then $E$ is called a \emph{liouvillian Picard-Vessiot extension} of $k.$ It is a standard fact that liouvillian Picard-Vessiot extensions are precisely those Picard-Vessiot extensions whose differential Galois groups have solvable identity components. 

\bprop\label{LPV}
	Let $E$ be a differential field extension of $k$ having $C$ as its field of constants. Then $E$ is a liouvillian Picard-Vessiot extension of $k$ if and only if  every  differential field $K$ intermediate to $E$ and $k$ is a solution field. 
\eprop

\begin{proof}
	Let $E$ be a liouvillian extension of $k,$ $K$ be an intermediate differential subfield and $\mathscr{H}$:= $\mathscr{G}(E|K).$ We know that $\gal^0$, the identity component of the differential Galois group $\gal$ of $E$ over $k$, is a solvable linear algebraic group. From \cite[Pages 6, 12]{Gro06}, we have the following facts:
	\begin{enumerate}[(i)]
		\item A closed subgroup $\mathscr H$ of an algebraic group $\gal$ is observable if and only if  $\galsub \cap \gal^0$ is observable in $\gal^0.$\\
		\item If $\gal$ is a solvable algebraic group then any closed subgroup $\galsub$ of $\gal$ is observable. 
	\end{enumerate}
	
	Since $\gal^0$ is solvable, every closed subgroup of $\gal^0$ is observable. Thus in particular $\galsub \cap \gal^0$ is observable in $\gal^0.$ This now implies that our  closed  subgroup $\galsub$ is observable. Now by Theorem \ref{observable}, $K$ must be a solution field. To prove the converse, we suppose that $E$ is not a liouvillian Picard-Vessiot extension of $k$. Then $\gal^0$ is not solvable and therefore it  contains a non-trivial Borel subgroup $\mathscr B.$ Since $\gal/\mathscr B$ is a projective variety,  we obtain from Theorem \ref{observable} that  the differential field $K$ corresponding to $\mathscr B$ is not a solution field. \end{proof}

The following structure theorem for intermediate differential subfields of  liouvillian Picard-Vessiot extensions will be used in the proof of our main theorem.
 
\begin{theorem} \label{LPV-structuretheorem} (\cite[Corollary 3.2]{URVRS})
 Let $E$ be a liouvillian Picard-Vessiot extension of $k$ for $M$ and the differential Galois group $\mathscr{G}:=\mathscr{G}(E|k)$ be connected. Let $K$ be a differential field intermediate to $E$ and $k$. Then $K$ is a solution field generated by a solution $\psi: N\to K$  having elements $t_1,\dots,t_n\in \psi(N)$ such that for each $i,$  $t'_i = a_it_i + b_i$ for $a_i \in k$ and $b_i \in k(t_1,\dots,t_{i-1})$. Furthermore
	\begin{enumerate}[(i)]
		
		\item if $\mathscr{G}$ is a unipotent algebraic group then we may assume $a_i=0$ for each $i$ and that $t_1,\dots, t_n$ are algebraically independent. \\
		
		\item If $\mathscr{G}$ is a torus then we may assume $b_i=0.$ 
		\end{enumerate}
	
\end{theorem}

\begin{proof}
	To avoid triviality, we shall assume $k\neq K.$ We know from Proposition \ref{LPV} that $K$ is a solution field. Let $\psi: N\to K$ be a $k[\partial]-$morphism and $y_1,\dots, y_m$ be a $k-$basis for the $k-$differential module $\psi(N).$ Then, each $y_i$ is a solution of some linear homogeneous differential equation over $k$ (see Proposition \ref{solutionfields-algebras}). Let $L$ be any differential field such that $k \subseteq L \subsetneq K$.  We claim that there is a $y \in \psi(N) \setminus L$ such that $y' = ay +b$ for some $a \in k,$ $b \in L$ and that $a$ can be taken zero if $\mathscr{G}$ is unipotent and that $b$ can be taken to be zero if $\mathscr{G}$ is a torus. 
	
	 Since $L\supsetneq K$, there is an $i$ such that $y_i\in \psi(N) \setminus L.$  Let $\mathscr{L} \in k[\partial]$ be of smallest positive degree $m$ such that $\mathscr{L}(y) = 0$ for some $y\in \psi(N)\setminus L$. Since $\mathscr{G}$ is connected and solvable, the differential operator is linearly reducible \cite[Theorem 2 of Section 22]{Kol48}. So, $\mathscr{L} = \mathscr{L}_{m-1} \mathscr{L}_1$, where $\mathscr{L}_{m-1}, \mathscr{L}_1 \in k[\partial]$ are of degrees $m-1$ and $1$ respectively. Let $\mathscr L_1 =  \partial - a$, $a \in k$. Observe that $\mathscr{L}_1(y) \in \psi(N)$ and $\mathscr{L}_{m-1}( \mathscr{L}_1(y)) = 0$. Therefore, from our choice of $m$, $\mathscr{L}_1(y) \in L$. Thus we have found an element $y \in \psi(N) \setminus L$ such that $y' = ay +b,$ where $a \in k$ and $b \in L$. 
	
	If $\mathscr{G}$ is unipotent then $E = k(\eta_1,\dots, \eta_n),$ where $\eta_i' \in k(\eta_1,\dots,\eta_{i-1})$ and in this case $\mathscr{L}$ admits a nonzero solution $\alpha \in k$ \cite[Proposition 2.2]{Sri20}. Thus  we may choose $\mathscr{L}_1=\partial -\left( \alpha'/\alpha \right).$ Then $y/\alpha \in \psi(N)\setminus L$ and $\left( y/\alpha \right) '=b/\alpha \in L$. Moreover, such $y/\alpha$ must be transcendental over $L$.
	
	Now suppose that $\mathscr{G}$ is a torus. Then $E= k\left( \xi_1,\dots, \xi_s\right) $, where $\xi_i'/\xi_i \in k$ for each $i$. If $b \neq 0$ then again from  \cite[Proposition 2.2]{Sri20} applied to the extension $L \left( \xi_1,\dots, \xi_s\right) $ of $L$ with $\mathscr{L}_1(y)=y'-ay=b,$ we obtain $\alpha \in L$ such that $\alpha'-a \alpha=b$. Now $y-\alpha\in \psi(N)  \setminus L$ and $(y-\alpha)'/(y-\alpha)= a \in k.$ This proves the claim.  
	
	Taking $L=k,$ we obtain an element $t_1\in \psi(N)$ and $t_1\notin k$ such that $t'_1=a_1t_1+b_1$ for $a_1,b_1\in k$. Likewise, we find elements $t_1,\dots, t_n\in \psi(N)$ and $t_i\notin k(t_1,\dots,t_{i-1})$ such that $t'_i=a_it_i+b_i,$ where $a_i\in k$ and $b_i\in k(t_1,\dots, t_{i-1})$. Since $\psi(N)$ is finite dimensional over $k$, there must be an integer $n$ such that  $\psi(N)\subset  k(t_1,\dots, t_n).$  Clearly, for that $n,$ we must have $K=k(t_1,\dots, t_n).$ 
 \end{proof}

\begin{remark}  \label{dihedral}\normalfont In the above theorem, the hypothesis that $\mathscr{G}\left( E|k\right) $ is connected allowed us to factor the differential operator $\mathscr{L}$ in $k[\partial]$, which is a crucial step in the proof. In fact, the assumption that $\mathscr{G}\left( E|k\right) $ is connected cannot be dropped. For example, consider the liouvillian extension $E=\mathbb{C}(x)\left( \sqrt{x}, e^{\sqrt{x}}\right),$ where	the derivation is $':=d/dx.$ Then $E$ is a liouvillian Picard-Vessiot extension of $\mathbb{C}(x)$ for the differential equation
		\begin{equation*}
			\mathscr{L}(y)= y''+\frac{1}{2x} y'-\frac{1}{4x}y=0.
		\end{equation*}
		The set $V:=\text{span}_{\mathbb{C}} \left\lbrace  e^{\sqrt{x}}, e^{-\sqrt{x}}\right\rbrace $ is the set of all solutions of $\mathscr{L}(y)=0$ in $E$. Since $E$ contains the algebraic
		extension $\mathbb{C}(x)\left( \sqrt{x}\right), \ \mathscr{G}\left( E|k\right) $  is not connected. One can show that the intermediate differential field $K:= \mathbb{C}(x)\langle e^{\sqrt{x}}+ e^{-\sqrt{x}} \rangle $ contains no elements satisfying a first order linear differential equation over $\mathbb{C}(x)$ other than the elements
		of $\mathbb{C}(x)$ itself \cite[p. 376]{Sri20}. \end{remark}



\section{Transcendence degree one subfields of strongly normal extensions} \label{trdone-SN}

\begin{proposition} \label{Riccati polynomials} Let $E$ be a differential field extension field of $k$ having $C$ as its field of constants. Let $L$ be a differential field intermediate to $E$ and $k.$ Suppose that there are two elements $y,t\in E$ each transcendental over $L$ such that $K=L(t)=L(y).$ Then $t$ satisfies a Riccati equation over $L$ if and only if $y$ satisfies a Riccati equation over $L.$
	\end{proposition}
\begin{proof} From L{\"u}roth's theorem,  we know that there are elements $a,b,c,d\in L$ such that $ad-bc\neq 0$ and that $$y=\frac{at+b}{ct+d}.$$ 

We compute the derivative of the above equation and obtain \begin{equation*}y'(ct+d)^2= \ ca't^2+(a'd+b'c)t+db' -\left( c'at^2+(bc'+ad')t +bd' \right) +(ad-bc)t'.\end{equation*}  If  $y'=f(y)$ is a polynomial of degree $\leq 2$ then we see that $y'(ct+d)^2$ is a polynomial of degree $\leq 2.$ Since $0\neq ad-bc\in L,$ we shall solve for  $t'$ and obtain $t'=g(t),$ where $g$ is a polynomial in one variable over $L$ of degree at most $2$. \end{proof}

We recall the following facts about  a strongly normal extension $E$ of $k.$  Since $\gal:=\gal(E|k)$ is an algebraic group, there is a chain of subgroups (see \cite[Proposition 12.2]{Kov06})
\begin{equation*}
	\gal \supseteq \gal^0 \supseteq \galsub \supseteq \left\lbrace 1\right\rbrace, 
\end{equation*}
where $\gal^0$ is the identity component of $\gal$, $\galsub$ is a connected linear algebraic group as well as a closed normal subgroup of $\gal^0$ such that $\gal^0/\galsub$ is an abelian variety. From the fundamental theorem of strongly normal extensions, we have a tower of fields 
\begin{equation}\label{decomp-SNE}
	k \subseteq E^0  \subseteq F \subseteq E,
\end{equation}where $E^0$ is a finite Galois extension of $k$ with Galois group $\gal/\gal^0$, $F$ is an abelian extension of $E^0,$ that is $\gal(F|E^0)\cong \gal^0/\galsub$ is an abelian variety and $E$ is a Picard-Vessiot extension of $F$ with $\galsub\cong \gal(E|F).$

\begin{theorem}\label{trdonesubfields-SNE}
	Let $E$ be a strongly normal extension of $k$ and $K$ be a differential field intermediate to $E$ and $k.$ If tr.deg$(K|k)=1$ then  there is a finite algebraic extension $\tilde{k}$ of $k$ and an element $t\in \tilde{k}K$ such that one of the following holds:
	
	\begin{enumerate}[(i)]
		\item $\tilde{k}K=\tilde{k}(t)$ and $t$ is a solution of a Riccati equation over $\tilde{k}.$ \\
		\item $\tilde{k}K=\tilde{k}(t,t')$ and $t$ is a solution of a  Weierstrass differential equation over $\tilde{k}.$
	\end{enumerate}

\end{theorem}

\begin{proof} 
	
Decompose $E$ into a tower of fields as in Equation (\ref{decomp-SNE}).  Fix an algebraic closure $\overline{E}$ of  $E$ and extend the derivation of $E$ to a derivation of $\overline{E}.$ Now we  split the proof into two cases.

Case (i): Suppose that $KF\supsetneq F$. Choose a finite algebraic extension $\tilde{k}\subset \overline{E}$ of $k$ so that there is a nonsingular projective curve $\Gamma$ defined over $\tilde{k},$ the function field $\tilde{k}(\Gamma)\cong \tilde{k}K$ and that $\Gamma$ has a $\tilde{k}-$point. We first claim that $\Gamma$ is a rational curve, that is, $\tilde{k}K=\tilde{k}(t).$ 

Since the compositum $\tilde{k}E$ remains a strongly normal extension of $\tilde{k}$ (see \cite[Theorem 5]{Kol53}),  for convenience of notation, we shall replace $\tilde{k}$ by $k.$  Choose an element $z\in KF\setminus F.$ Since  $\gal(E|F)$ is a connected linear algebraic group  and since a connected linear algebraic group is a union of its Borel subgroups, by the fundamental theorem, we obtain that $$\bigcap_{\mathscr B, \ \text{Borel subgroups}}E^{\mathscr B}=F.$$ Thus, there is a Borel subgroup $\mathscr B$ of $\gal(E|F)$ such that $z\in E\setminus E^{\mathscr B}.$  Note that $\mathscr{B}=\mathscr U\rtimes \mathscr T,$ where $\mathscr U$ is the unipotent radical and $\mathscr T$ is a torus. Thus $E$ can be further decomposed as follows: \begin{equation*}E\supseteq E^{\mathscr U}\supseteq E^{\mathscr B}\supseteq F,\end{equation*}  
\begin{enumerate} [(i)] \item $E$ is a Picard-Vessiot extension of  $E^{\mathscr{U}}$ whose Galois group $\mathscr U$ is unipotent.\\
		\item $E^{\mathscr U}$ is a  Picard-Vessiot extension of $E^{\mathscr B}$ whose Galois group is isomorphic to the torus $\mathscr T\cong\mathscr B/\mathscr{U}.$
	\end{enumerate}

Suppose that the compositum $E^{\mathscr U}K$ properly contains $E^{\mathscr U}.$ Then since $\mathscr U$ is connected, $E^{\mathscr U}$ is algebraically closed in $E^{\mathscr U}K$ and we obtain that $\mathrm{tr.deg}(E^\mathscr{U}K| E^\mathscr U)=1.$ Now we apply  Theorem \ref{LPV-structuretheorem} and obtain $E^\mathscr UK=E^\mathscr U(y),$ where $y'\in E^{\mathscr U}.$ This, in particular, shows that the (geometric) genus of $\Gamma$ is $0.$ Since $\Gamma$ has  $k-$point, we obtain that $K=k(t)$ (see \cite[Sections 2 and 6 of Chapter 2] {Che51}).

Thus $E^{\mathscr U}K=E^{\mathscr U}(t)=E^{\mathscr U}(y).$   From Proposition \ref{Riccati polynomials}, we then have $t$ is a solution of a Riccati equation over $E^{\mathscr U}.$ Say, $t'=g(t),$ where $g$ is a polynomial over $E^\mathscr U$ of degree $\leq 2.$  Since $K=k(t)$ is a differential field, we have $t'=h_1(t)/h_2(t),$ where $h_1$ and $h_2$  are relatively prime polynomials in $k[t]$ with $h_2$ being a monic polynomial. Then $h_1$ and $h_2$ must remain relatively prime over any field extension of $k$. Now since $g(t)h_2(t)=h_1(t)$,  we must have $h_2=1$ and we obtain that the coefficients of $g$ are in $k$.

Suppose that $E^{\mathscr U}K=E^{\mathscr U},$ that is $K\subseteq E^{\mathscr U}.$ Since $\mathscr T$ is a commutative group, $KE^\mathscr B$ must be a Picard-Vessiot extension of $E^\mathscr B$ and since $\mathscr T$ is connected, $E^\mathscr B$ is algebraically closed in $E^{\mathscr U}$ and therefore in $KE^\mathscr B$ as well. Now $\mathrm{tr.deg}(KE^\mathscr B|E^\mathscr B)=1$ implies that the differential Galois group of $KE^\mathscr B$ over $E^\mathscr B$ is isomorphic to $\mathrm{G_m}$ and it follows that $E^\mathscr BK=E^\mathscr B(y)$ such that $y'/y\in E^\mathscr B$ (see \cite[Example 5.24]{Mag94}). Thus we have again shown that $\Gamma$ is a genus zero curve. Therefore $K=k(t)$ and we obtain $E^{\mathscr B}K=E^{\mathscr B}(t)=E^{\mathscr B}(y)$ with  $y'/y\in E^{\mathscr B}.$  Now a similar argument, as in the previous paragraph, shows that $t$ satisfies a Riccati equation over $k.$ This completes the proof of the theorem in this case.

Case (ii). Suppose that $K\subseteq F.$ Then since $F$ is an abelian extension of $E^0,$  every intermediate differential subfield is strongly normal over $E^0.$ In particular, $KE^0$ is  strongly normal over $E^0$ and  $\mathrm{tr.deg}(KE^0|E^0)=1.$  Now we apply \cite[Theorem 3]{Kol55} and obtain that either $KE^0=E^0(y),$ where $y'\in E^0$ or $y'/y\in E^0\setminus \{0\}$ or that $\bar{k}K=\bar{k}(t,t'),$ where $t$ is a solution of a  Weierstrass differential equation over $\bar{k}.$ If $KE^0=E^0(y)$ with  $y'\in E^0$ or $y'/y\in E^0\setminus \{0\}$  then $\gal(KE^0|E^0)\cong \mathrm{G_a}(C)$ or $\mathrm{G_m}(C)$ and in any event, we must then have a surjective morphism from the abelian variety $\gal(F|E^0)$ to the linear algebraic group $\gal(KE^0|E^0),$ which is impossible.  Therefore, $\bar{k}K=\bar{k}(t,t'),$ where $t$ is a nonalgebraic solution of a  Weierstrass differential equation over $\bar{k}.$  Since $K$ is finitely generated over $k,$  there is a finite algebraic extension $\tilde{k}$ of $k$ such that $\tilde{k}K=\tilde{k}(t,t').$  
\end{proof}

\section{Transcendence degree one subfields of iterated strongly normal extensions}\label{trdone-ISNE}

	\begin{proposition} \label{SNE-algebraic}  
	Let $E$ be a strongly normal extension of $k$ with differential Galois group $\gal$ and $\tilde{E}$ be an algebraic extension of $E$. Suppose that  $K$ is a differential field intermediate to $k$ and $\tilde{E}$ such that  $\mathrm{tr.deg}\left( K|k\right) =1$.  Then $\mathrm{tr.deg}(K\cap E | k)=1$ and there is an algebraic extension $\tilde{k}$ of $k$  such that  one of the following holds:
	\begin{enumerate}[(i)]
		\item $\tilde{k}(K\cap E)=\tilde{k}(t),$  where $t$ satisfies a Riccati differential equation over $\tilde{k}$. \\
		\item  $\tilde{k}(K\cap E)=\tilde{k}(t,t'),$ where $t$ satisfies a Weierstrass differential equation over $\tilde{k}$. 
	\end{enumerate} 
\end{proposition}

\begin{proof}  From \cite[Theorem 5]{Kol53}, we have that $KE$ is a strongly normal extension of $K$ and that the natural restriction map from the differential Galois group $\gal(KE|K)$ to $\gal$ is an isomorphism  onto the subgroup $\gal(E| K\cap E)$ of $\gal$.  In particular, if $\mathrm{tr.deg}(E|k)=n$ then $\mathrm{tr.deg}(E| K\cap E)=$ dim$(\gal(E| K\cap E))=$ $\mathrm{tr.deg}(KE|K)=n-1.$ Thus, $\mathrm{tr.deg}((K\cap E)|k)=1.$ Now,  the rest of the proof follows from Theorem \ref{trdonesubfields-SNE}. 
\end{proof}

\begin{lemma} \label{descent-lemma} Let $L$ be a differential field extension of $\bar{k}$ having $C$ as its field of constants. Suppose that $K$ and  $E$ are  differential subfields  intermediate to $L$ and $\bar{k}$ having the following properties: $K$ is finitely generated over $\bar{k},$ $\mathrm{tr.deg}(K|\bar{k})=1,$ $\mathrm{tr.deg}(EK|E)=1,$ $E$ is an algebraically closed field, $EK$ is weakly normal over $K$ and that the group of differential automorphisms $\gal(EK|K)$ stabilises $E.$ \begin{enumerate}[(i)] \item \label{Riccati descent} If there is a Riccati equation over $E$ having a  solution $t \in EK \setminus E$  then there is  Riccati equation over $\bar{k}$ having a solution $v \in K\setminus \bar{k}.$\\
		
	\item \label{weierstrassiandescent} If there is a Weierstrass differential equation over $E$ having a solution $t\in EK\setminus E$  then there is a Weierstrass differential equation over $\bar{k}$ having a  solution $v\in K\setminus \bar{k}.$
\end{enumerate}

\end{lemma}

\begin{proof}  From our assumptions that $\gal(EK|K)$ stabilizes $E,$ $\mathrm{tr.deg}(EK|E)=1$  and that $t$ is transcendental over $k,$ we have for each $\s\in \gal(EK|K)$,   $\s(t)$ is algebraic over $E(t).$ Since $K$ is finitely generated over $\bar{k},$ 
$EK$ is finitely generated over $E$ and we conclude that  there are finitely many automorphisms $\s_1,\dots, \s_n\in \gal(EK|K),$ where $\s_1$ is the identity,  such that $\gal(EK|K)$ stabilizes the differential field $E^*:=E(\s_1(t),\s_1(t)',\dots,$ $\s_n(t), \s_n(t)').$ 

We now claim that $(E^*\setminus E)\cap K\neq \emptyset.$  Since $\mathrm{tr.deg}(K|\bar{k})=$ $\mathrm{tr.deg}(EK|E)=1,$  there is an element $u\in K\setminus \bar{k}$ such that $u$ is transcendental over $E.$  	Let $X^m+\a_{m-1}X^{m-1}+\cdots+\a_0\in E^*[X]$ be the irreducible polynomial of $u$ over $E^*$. 
	
\begin{multicols}{2}

Then for any $\sigma \in \gal(EK|K)$, we have $$u^m+\sigma( \a_{m-1})  u^{m-1}+\cdots+\sigma( \a_0) =0$$ and thus	\begin{equation*}
		\left(\s(\a_{m-1})-\a_{m-1}\right)  u^{m-1}+\cdots+\sigma(\a_0)-\a_0 =0.
	\end{equation*} Therefore  $\sigma(\a_i) =  \a_i$ for all $i=0,\dots, m-1.$ Since $EK$ is assumed to be weakly normal over $K,$ we obtain that $\a_i \in K$. That is $\a_i\in E^*\cap K$ for all $i.$ Now $u$ is transcendental over $E$ implies that there is at least one $i$ such that $s:=\a_i\in E^*\setminus E$. This proves the claim.

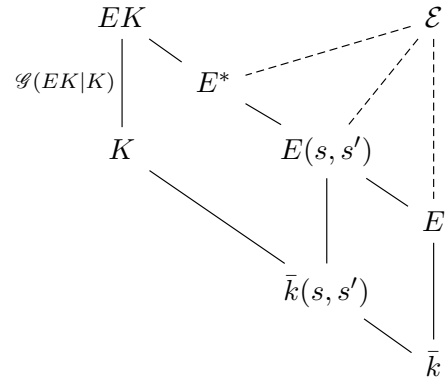
\columnbreak

\hspace{2.4em}\begin{tikzcd}[row sep=1.6ex, column sep= 1.6ex]
	EK\arrow[dash]{dd}[swap]{\gal(EK|K)}\arrow[dash]{dr}& & &\mathcal E\arrow[dash, dashed]{dll}\arrow[dash, dashed]{ddl}\arrow[dash, dashed]{ddd}\\
	&E^*\arrow[dash]{dr}&& \\
K\arrow[dash]{drdr}& & E(s,s')\arrow[dash]{dr}\arrow[dash]{dd}&\\
		&&&E\arrow[dash]{dd}\\ 
	&&\bar{k}(s,s')\arrow[dash]{dr}&\\
&	& & \bar{k}
\end{tikzcd}
\end{multicols}

	Now we are ready to prove the lemma. Let  $t'=P(t)$ be a Riccati equation over $E$ with $t\in EK\setminus E.$ Then for  any $\s\in\gal(EK|K),$ $\s(t)'=P_\s(\s(t))$ is also a Riccati equation, where $P_\s\in E[X]$ is the polynomial obtained by applying $\s$ to the coefficients of $P.$
	Since each of $\s_1(t), $ $\dots, \s_n(t)$ is a solution of some  Riccati equation over $E,$ one can construct  a  Picard-Vessiot extension $\mathcal E$ of $E$ containing $\s_1(t),\dots, \s_n(t).$ Now $E(s,s')$ is a differential subfield of $\mathcal E$ with $\mathrm{tr.deg}(E(s,s')|E)=1.$  Applying Theorem \ref{trdonesubfields-SNE},  we obtain $E(s,s')=E(y),$ where $y$ satisfies a Riccati  differential equation over $E$. Then $\bar{k}(s,s')$ has genus $0$ and thus $\bar{k}(s,s')=\bar{k}(v).$  Applying Proposition \ref{Riccati polynomials}, we obtain that $v$ satisfies a Riccati  equation over  $\bar{k}.$ This proves (\ref{Riccati descent}).
	
	Suppose that $t$ is a nonalgebraic solution of a  Weierstrass differential equation over $E;$ $t'^2=\a^2(4t^3-g_2t-g_3)$ with $g_2,g_3\in C.$  In this case, we claim that  the differential field $\overline{k}(s,s')$ contains a nonalgebraic solution of a  Weierstrass differential equation over $\bar{k}$, which would then complete the proof of this lemma.   Observe that  $\s(t)'^2=\s(\a)^2(4\s(t)^3-g_2\s(t)-g_3)$ for $\s\in \gal(EK|K).$ Then $E^*$, being a compositum of finitely many strongly normal extensions $E(\s_i(t), \s_i(t)')$ of $E,$  is a strongly normal extension of $E.$ Also, $E^*\subseteq EK$ and therefore $\mathrm{tr.deg}(E^*|E)=1.$   From the facts that  $E(t,t')$ is a genus one field extension of $E$ and that $E(t,t')\subseteq E^*,$  we infer that $E^*$ must be  a genus one abelian  extension of $E.$ Now we consider the  differential subfield $E(s,s')$ of $E^*.$   As  noted in the proof of case (ii) of Theorem \ref{trdonesubfields-SNE},  $E(s,s')$ is also a genus one abelian extension of $E$. This implies $\bar{k}(s,s')$ is a genus one field extension of $\bar{k}.$ 
	
	Let $\mathscr C_1$ be a nonsingular projective model for $\bar{k}(s,s').$ Then $\mathscr C_1$ is also a model for $E(s,s').$
	By \cite[Theorem 3]{Kol55}, the group $\gal(E(s,s')|E)$ is the $C-$points of a Weierstrass  elliptic curve $\mathscr C_2$ defined over $C:$  $$\mathscr C_2=X^2_2X_0-4X^3_1+h_2X^2_0X_1+h_3X^3_0,$$  where $h_2,h_3\in C.$ Furthermore, $E(s,s')=E(x, x'),$ where $(1:x:x'/\alpha)$ is a point of $\mathscr C_2$ for some $\alpha\in E.$ Thus $ \mathscr C_2$ is also a nonsingular  projective model for $E(s,s').$ Then, $\mathscr C_1$ and $\mathscr C_2$ are isomorphic over $E.$ Since both the curves are defined over $\bar{k}$ and the $ j-$ invariants $j(\mathscr C_1)=j(\mathscr C_2)\in C\subseteq \bar{k},$ the curves are isomorphic over $\bar{k}$ as well \cite[Proposition 1.4]{Sil-2009}. Thus, there is a $\bar{k}(s,s')-$point $(1:\omega:\rho)$ of $\mathscr C_2$ such that $\bar{k}(s,s')=\bar{k}(\omega, \rho).$ Then, as in the proof of \cite[Theorem 3]{Kol55}, one can show that $\bar{k}(s,s')=\bar{k}(\omega, \omega')$ and that $(1:\omega:\omega'/\beta)$ is a point on $\mathscr C_2$ for some $\beta\in\overline{k}.$ This implies $\omega'^2=\beta^2(4\omega^3-h_2\omega-h_3).$  \end{proof}

\bt \label{maintheorem} Let $E$ be an iterated strongly normal extension of $k$ and $K$ be an intermediate differential field having field transcendence degree $\mathrm{tr.deg}(K|k)=1.$  Then, there is a finite algebraic extension $\tilde{k}$ of $k$ and an element $t\in \tilde{k}K$ such that one of the following holds:

\begin{enumerate}[(i)]
	\item $\tilde{k}K$ is a finite algebraic extension of $\tilde{k}(t)$ and $t$ is a solution of a Riccati differential equation over $\tilde{k}.$ \\
	\item $\tilde{k}K$ is a finite algebraic extension of $\tilde{k}(t,t')$ and $t$ is a solution of a  Weierstrass differential equation over $\tilde{k}.$
\end{enumerate}

\et

\begin{proof} Let  $k=:E_0\subseteq E_1\subseteq E_2\subseteq \cdots \subseteq E_{n+1}=E$ be a tower of strongly normal extensions, where we may assume that $E_{n+1}$ is an algebraic extension of $E_n$ and that $E_n$ is a transcendental extension of $E_{n-1}.$ Let $\overline{E}$ be an algebraic closure of $E.$ Consider the following tower of algebraically closed differential fields \begin{equation*} \overline{E}_0\subseteq \overline{E}_1\subseteq \overline{E}_2\subseteq\cdots\subseteq \overline{E}_n=\overline{E} .\end{equation*}  
Since the differential field $\overline{E}_{n-1}K$ is contained in $\overline{E},$ by Proposition \ref{SNE-algebraic},  we obtain an element $y\in \overline{E}_{n-1}K$ transcendental over $\overline{E}_{n-1}$ such that $y$ satisfies either a Riccati or a Weierstrass differential equation over $\overline{E}_{n-1}.$ 

For the moment, let us assume that   \begin{itemize} \item[(C1)]\label{assumptionC}for all $i=1,\dots, n-1,$ $\overline{E}_iK$ is a weakly normal extension of $\overline{E}_{i-1}K$ and that the group $\gal(\overline{E}_iK|\overline{E}_{i-1}K)$ stabilises $\overline{E}_i.$\end{itemize} The differential field $K$, being a subfield of an iterated strongly normal extension, is  finitely generated over $k.$ Therefore, $\overline{E}_iK$ is finitely generated over $\overline{E}_i.$  Now, a repeated application of Lemma \ref{descent-lemma}  will produce a solution $t\in \bar{k}K\setminus \bar{k}$  of a Riccati or a Weierstrass differential equation over $\bar{k}.$ Indeed, one only needs a finite algebraic extension $\tilde{k}$ so that the differential field $\tilde{k}K$ contains both $t$ and the coefficients of the differential equation of $t$. Thus the theorem is proved under the assumption (C1).

Now we shall verify the validity of (C1). Note that $\overline{E}_{i-1}E_i$ is strongly normal over $\overline{E}_{i-1}.$ Therefore the $L^*:=\overline{E}_{i-1}KE_i$ is a strongly normal extension of $L:=K\overline{E}_{i-1}.$ Since  $\overline{E}$ and $k$ have the same field of constants, 
the group $\gal(L^*|L)$ stabilizes the strongly normal extensions $E_i$ of $E_{i-1}$ and $\overline{E}_{i-1}E_i$ of $\overline{E}_{i-1}.$ Next, we consider the set $I$ of all irreducible polynomials whose coefficients  are in $E_i.$ The group $\gal(L^*|L)$ acts on $I$ via the map $P\mapsto P_\s,$  where $P_\s$ is obtained by applying automorphism $\s$ to the coefficients of $P.$  Now we lift each $\s\in \gal(L^*|L)$ to a differential automorphism of the splitting field $L^*(I)$ of $I$ contained in $\overline{E}.$ Observe that $L^*(I)=L^*\overline{E}_i=K\overline{E_i}.$ Since the fixed field of the group $\gal(K\overline{E}_i|L^*)$ is the field $L^*$ and that the fixed field of $\gal(L^*|L)$ is $L=\overline{E}_{i-1}K,$ we obtain that the fixed field of $\gal(\overline{E}_iK|\overline{E}_{i-1}K)$ is $\overline{E}_{i-1}K.$ This shows that $\overline{E}_iK$ is weakly normal over $\overline{E}_{i-1}K.$ Every element of $\gal(\overline{E}_iK|\overline{E}_{i-1}K)$ stabilizes  the strongly normal extension $E_i$ of $E_{i-1}$ as well as the set $I$ of polynomials over $E_i.$ Therefore the group stabilizes $\overline{E}_i.$ This shows that (C1) always holds. \end{proof}

\begin{corollary} \label{SNEoverC}
	If $E$ is an iterated strongly normal extension of the field of constants $C$ then every intermediate differential field $K$ with $\mathrm{tr.deg}(K|C)=1$ has an element $t$ such that either $t'=1$ or $t'=ct$ for some nonzero $c\in C$ or $t$ is a solution of a Weierstrass differential equation over $C.$
	\end{corollary}

\begin{proof}
	We only need to consider the case when there is an element $y\in K\setminus C$ such that  $y'=ay^2+by+c$ for $a,b,c\in C.$ In this case, the differential field $C(y)$ can be embedded in a Picard-Vessiot extension of $C.$ Now since any Picard-Vessiot extension of $C$ has a connected abelian differential Galois group, we obtain that $C(y)$ itself must be a Picard-Vessiot extension of $C$ having a differential Galois group isomorphic to $\mathrm{G_a}(C)$ or $\mathrm{G_m}(C).$ It follows that $C(y)=C(t),$ where either $t'=1$ or $t'=ct$ for some nonzero constant $c.$  \end{proof}

\begin{proof}[Proof of Theorem \ref{intromaintheorem}]  Follows from Theorem \ref{trdonesubfields-SNE}, Theorem \ref{maintheorem} and Corollary \ref{SNEoverC}.
\end{proof}


\section{nonalgebraic solutions of first order differential equations} \label{application}
\subsection{Differential equations of general type} In the next theorem, we device a mechanism to generate differential equations of general type.
\begin{theorem} \label{generaltypegeneration}
	Let $f\in \bar{k}[Y,Z]$ be an irreducible polynomial having the following properties: \begin{enumerate}[(a)]\item $p:=(0,0)$ is a simple point of $f$ and $Z$ is the tangent line at $p.$ \\
	\item With respect to the uniformizing parameter $Y,$ both the coefficients $\lambda_2$ and $\lambda_3$ of the $Y-$adic expansion $Z=\lambda_2Y^2+\lambda_3Y^3+\cdots$   do not have any antiderivatives in $\bar{k}.$ 
	\end{enumerate}
Then the differential equation $f(y,y')=0$  is of general type.
\end{theorem}

\begin{proof}
	Consider $\bar{k}(f)=\bar{k}(y,z)$  with the induced derivation  $': \bar{k}(y,z)\to \bar{k}(y,z);$ $y'=z.$  In order to show that $f(y,y')=0$ has  a nonalgebraic solution, we only need to show that the field of constants of $\bar{k}(y,y')$ is the same as that of $\bar{k}.$ Since $\bar{k}(y,y')$ embeds as a differential field in $\bar{k}((y)),$   for an element $w\in \bar{k}(y,y')\setminus \bar{k},$  we can write the $y-$adic expansion $ w=\sum_{i=r}^{\infty} a_i y^i,$ $a_i\in \bar{k}$ and $a_r\neq 0$ and obtain
\begin{align}
		w'&=a'_ry^r+a'_{r+1}y^{r+1}+\cdots+ra_ry^{r-1}(\lambda_2y^2+\lambda_3y^3+\cdots)+(r+1)a_{r+1}y^r(\lambda_2y^2+\cdots) +\cdots\notag \\ &= a'_ry^r+(a'_{r+1}+ra_r\lambda_2)y^{r+1}+(a'_{r+2}+ra_r\lambda_3+(r+1)a_{r+1}\lambda_2)y^{r+2}+\cdots. \label{zprimeexpression}
	\end{align} 

We shall first prove that 
\begin{itemize}
\item[(C2)]\label{nofirstordereqn}there is no $w\in \bar{k}(y,y')\setminus \bar{k}$ such that $w'=\alpha w+\beta$ for any $\alpha, \beta\in \bar{k}.$ \end{itemize} Suppose not. Then, substituting $w=\sum^\infty_{i=r}a_ry^r,$ we have the following cases to consider:

Case (i). ord$_{p}(w)=0$.   Here we have $a'_0=\alpha  a_0+\beta$ and if  $m\geq 1$ is the least positive integer such that $a_m\neq 0$ then we have $a'_m=\alpha a_m$ and that $a'_{m+1}+ma_m\lambda_2=\alpha a_{m+1}.$ Thus, we obtain $$\left(-\frac{a_{m+1}}{ma_m}\right)'=\lambda_2,$$ a contradiction to our assumption on $\lambda_2$.

Case (ii). ord$_{p}(w)\geq 1$ or ord$_p(w)\leq -2.$   Then $a'_r=\alpha a_r$ and that  $a'_{r+1}+ra_r\lambda_2=\alpha a_{r+1} .$ This implies $$\left(-\frac{a_{r+1}}{ra_r}\right)'=\lambda_2$$ and as before we obtain a contradiction.

Case (iii). ord$_p(w)=-1.$ Then  $a'_{-1}=\alpha a_{-1}$ and $a'_1=\alpha a_1+\lambda_3 a_{-1}.$ This implies $$\left(\frac{a_1}{a_{-1}}\right)'=\lambda_3,$$ which contradicts our assumption on $\lambda_3.$

Thus we have shown that there is no element $w\in \bar{k}(y,y')\setminus \bar{k}$ such that $w'=\alpha w+\beta$ for any $\alpha, \beta\in \bar{k}.$ In particular, we could choose $\alpha=\beta=0$ to show that the field of constants of $\bar{k}$ is the same as the field of constants of $\bar{k}(y,y').$ 

We shall now show that the equation $f(y,y')=0$ is of general type. Suppose  that there is an element $w\in \bar{k}(y,y')\setminus \bar{k}$ such that $w'=b_2w^2+b_1w+b_0$ for $b_2,b_1,b_0\in \bar{k}.$ Then  we must have $b_2\neq 0.$ If ord$_p(w)\leq -1$ then from Equation (\ref{zprimeexpression}) we obtain ord$_p(w')\geq$ ord$_p(w).$ But $b_2\neq 0$ and ord$_p(w)\leq -1$ implies ord$_p(w')=$ ord$_p(b_2w^2+b_1w+b_0)=2$ ord$_p(w).$ Thus we have obtained ord$_p(w)\leq$ ord$_p(w')=2$ ord$_p(w),$ a contradiction.  Next, if ord$_p(w) \geq 1$ then   $b_0=0$ and we obtain that $(1/w)'=-b_1(1/w)-b_2,$ which contradicts (C2).  Finally, if ord$_p(w)=0$ then  $a_0\in \bar{k}$ is a solution of the Riccati equation $a'_0=b_2 a^2_0+b_1a_0+b_0.$ A quick calculation shows  $$\left(\frac{1}{w-a_0}\right)'= \frac{-2b_2a_0-b_1}{w-a_0}-b_2$$ and this again contradicts (C2). Thus we have shown that there is no element $w\in \bar{k}(y,y')\setminus \bar{k}$ satisfying a Riccati equation over $\bar{k}.$ 

Now we suppose that there is a transcendental element $w\in \bar{k}(y,y')$ satisfying a Weierstrass differential equation:   $w'^2=\alpha^2(4w^3-g_2w-g_3)$ for $g_2,g_3\in C$ and $\alpha \in \bar{k}.$  Then from the $y-$adic expansion of $w$ and from Equation (\ref{zprimeexpression}), we have $w'^2=a'^2_ry^{2r}+\cdots$ and thus 
\begin{equation*}2\ \text{ord}_p(w)\leq \text{ord}_p(w'^2)= \text{ord}_p(\alpha^2(4w^3-g_2w-g_3)).\end{equation*}

If  ord$_p(w)< 0$ then ord$_p(\alpha^2(4w^3-g_2w-g_3))= 3\ \text{ord}_p(w)$ and we obtain  $2$ ord$_p(w)\leq 3$ ord$_p(w),$   a contradiction. If ord$_p(w)>0$ then ord$_p(\alpha^2(4w^3-g_2w-g_3))\leq$ ord$_p(w),$  and we obtain $2 \ \text{ord}_p(w)\leq$ ord$_p(w),$    again a contradiction.  Finally, let ord$_p(w)=0$ and in the $y-$adic expansion of $w,$ let $m$ be the least positive integer such that $a_m\neq 0.$ Then from Equation (\ref{zprimeexpression}),  we have the following equations
\begin{align} w'^2&=\alpha^2(4w^3-g_2w-g_3) \notag \\ a'^2_0&=\alpha^2(4a^3_0-g_2a_0-g_3)\notag\\ 2a'_0a_m'&=\alpha^2\left(12a^2_0a_m-g_2a_m\right)\label{coeff-y^m}. \end{align}
We assume for the moment that $a_0$ is a constant. Then $a_0$ must be one of the distinct roots of the polynomial $4Y^3-g_2Y-g_3$ and in particular, $a_0$ is not a root of $12Y^2-g_2.$ Now Equation (\ref{coeff-y^m}) becomes $$\alpha^2(12a^2_0-g_2)a_m=0,$$ which is absurd. Thus there is no element in $\bar{k}(t,t')\setminus \bar{k}$ satisfying a Weierstrass differential equation over $\bar{k}.$

Now we shall show that $a_0$ is indeed a constant. Assume otherwise and consider the nonsingular projective curve \begin{equation} \label{weierstrassianform}X^2_2X_0-4X^3_1+g_2X^2_0X_1+g_3X^3_0.\end{equation} We have two nonconstant points of this curve, namely, $(1:w:w'/\alpha)$ and $(1:a_0:a'_0/\alpha).$ If $(1:\eta:\xi):=(1:w:w'/\alpha)(1:a_0:a'_0/\alpha)$ then $$\eta=-w-a_0+\frac{1}{4\alpha^2}\left(\frac{w'-a'_0}{w-a_0}\right)^2$$ 
and thus $\eta\in \bar{k}(y,y')\setminus \bar{k}.$ Now we apply \cite[Lemma 2]{Kol53}  for the Weierstrass equations 
$w'^2=\alpha^2(4w^3-g_2w-g_3)$ and  $a'^2_0=(-\alpha)^2(4a^3_0-g_2a_0-g_3)$ and obtain that $\eta'=0.$ This contradicts the fact that field of  constants of $\bar{k}(y,y')$ is the same as the field of constants of $\bar{k}.$
\end{proof}

\begin{remark} \normalfont If we have an irreducible  affine curve of the form $Z-F_2-F_3-\cdots-F_n,$ where $F_i$ are forms of degree $i,$ then one can quickly find the coefficients $\lambda_2$ and $\lambda_3$ of the $Y-$adic expansion of $Z$ as follows:  Let $F_2=x_{20}Y^2+x_{11}ZY+x_{02}Z^2 $ and $F_3=x_{30}Y^3+R_3.$ Then, on the curve, the value of  $Z$ equals  \begin{equation*} x_{20}Y^2+x_{11}ZY+x_{02}Z^2+x_{30}Y^3+R_3+F_4+\cdots+F_n.\\
	\end{equation*}
In the above expression we shall substitute back for $Z$  and obtain
\begin{equation*} 
x_{20}Y^2+x_{11}\left(x_{20}Y^2+x_{11}ZY+\cdots\right)Y+x_{02}\left(x_{20}Y^2+x_{11}ZY+\cdots\right)^2+x_{30}Y^3+\cdots.\\
\end{equation*}
Continuing this process one actually obtains the $Y-$adic expansion of $Z;$  $$Z=x_{20}Y^2+\left(x_{11}x_{20}+x_{30}\right)Y^3+\cdots.$$
	\end{remark}

\begin{example}Let $\C(x)$ be the ordinary differential field of rational functions with the derivation $x'=1.$  Then differential equation $$y'-\frac{1}{x}y^2-xyy'-\frac{1}{x+1}y^3+y(y')^2=0$$ is of general type as $1/x$ and $x(1/x)+1/(x+1)=1+1/(x+1)$ have no antiderivatives in $\C(x).$
	\end{example}

\begin{example}\label{Abel-DE} The differential equation \begin{equation}\label{genabel}y'=a_ny^n+\cdots+a_3y^3+a_2y^2, \end{equation} 
	where both $a_2$ and $a_3$ having no antiderivatives in $k,$ is  readily seen to be of general type.
Thus, the differential equation $y'=y^3-y^2$  is of general type over $\C.$   In fact, any autonomous equation of the kind (\ref{genabel}) with $a_2$ and $a_3$ nonzero is of general type.
\end{example}

\begin{remark}\label{rational-autoclass}   \normalfont Let $C$ be an algebraically closed field with the zero derivation. From Theorem \ref{intromaintheorem} and from \cite[Proposition 3.1]{Sri17}, we have the following equivalent statements: 
	\begin{enumerate}[(i)]\label{autonomous-equivalence}
	\item An autonomous differential equation over $C$ of the form $y'=f(y);\ f\neq 0$ is not of general type, that is, the equation has a nonalgebraic solution in an iterated strongly normal extension of $C.$\\
	\item The differential field $C(y)$ has an element $z$ such that either $z'=1$ or $z'=cz$ for some nonzero constant $c.$ \\
	\item There is an element $z\in C(y)$ and a nonzero constant $c\in C$ such that $$\frac{1}{f(y)}=\frac{\partial z}{\partial y}\quad \text{or}\quad \frac{1}{f(y)}=\frac{1}{cz}\frac{\partial z}{\partial y}.$$ 
	\end{enumerate}
	Thus the autonomous equation $y'=f(y)$ is not of general type if and only if  either $1/f(y)$ has no residues at any element of $C$, that is, the partial fraction expansion of $1/f(y)$ is of the form \begin{equation} \label{auto-antideriv} h(y)+\sum_{i=1}^{n} \sum_{j=2}^{n_i}\frac{d_{i j}}{(y-c_i)^j},\end{equation}  or the partial fraction expansion of $1/f(y)$ is of the form \begin{equation}\label{auto-exponential}c\sum_{i=1}^{n} \frac{m_i}{(y-c_i)},\end{equation} where $m_i$ are nonzero integers and $c$ is a nonzero constant.
	
	The equations $y'=y^3-y^2$ and $y'=\dfrac{y}{y+1}$ are now easily seen to be of general type.
	\end{remark}

\subsection{Differential equations of nongeneral type.} \label{algebraicdependenceofsolutions} In this section we shall study the algebraic dependence of first order differential equations of nongeneral type.  Let $L$ be a differential field extension of $k$ having $C$ as its field of constants.  It is easily seen that any two nonzero solutions in $L$ of a homogeneous first order linear differential equation $y'=by,$ where $0\neq b\in k,$ are $C-$linearly dependent. Suppose that $L$ has  nonalgebraic solutions of the nonhomogeneous  differential equation $y'=by+c$ for $b,c\in k$ and $c\neq 0.$ Consider a Picard-Vessiot extension $E$  of $L$ for $M=k[\partial]/k[\partial]\mathscr L,$ where $\mathscr L=\partial^2-(b+(c'/c))\partial+b(c'/c)-b'$ is obtained by  homogenizing $y'=by+c.$  Let $V\subset E$ be the set of all solutions of $\mathscr L(y)=0.$ Then dim$_CV=2.$  The differential field $\mathcal E:=k\langle V\rangle,$  generated by elements of $V$ and all its derivatives, is a Picard-Vessiot extension of $k$ (for $M$).  For any $u\in L$ such that  $u'=bu+c,$ we see that $\mathscr L(u)=0$ and therefore $u\in V\subset \mathcal E.$  Moreover, for any automorphism $\s\in \gal(\mathcal E | k),$ we have $\mathscr L(\s(u))=\s\left(\mathscr L(u)\right)=0$ and that $(\s(u)-u)'=b(\s(u)-u).$ Thus, if we choose an automorphism $\s\in \gal(\mathcal E| k)$ such that $\s(u)\neq u$ then since $\mathscr L(\s(u)-u)=\mathscr L (\s(u))-\mathscr L(u)=0,$ we have $\{\s(u)-u, u\}$  to be a $C-$basis of $V.$ We therefore have $\mathcal E=k(\s(u), u),$ where the fields $k(\s(u)-u)$ and $k(u)$ are differential fields. This implies $\mathrm{tr.deg}(\mathcal E|k)\leq 2$ and thus  any three  solutions in $L$ of $y'=by+c,$ where $b,c\in k$ and $c\neq 0,$ must be $k-$algebraically dependent.

Suppose that  \begin{equation}\label{alg-indp-Riccati} y'=ay^2+by+c \ \text{for} \ a,b,c\in k \ \text{with}\ a\neq 0 \end{equation} has a  nonalgebraic solution  in $L.$  Let $$\mathscr L= \partial^2-\left(\frac{a'}{a}+b\right)\partial+ac$$ and $E$ be a Picard-Vessiot extension of $L$ for $M=k[\partial]/k[\partial]\mathscr L.$  Let $V\subset E$ be the set of all solutions of  $\mathscr L(y)=0.$ Then  $v\in V\setminus \{0\}$ if and only if $-v'/av$ is a solution of the Riccati equation (\ref{alg-indp-Riccati}). Note that $\mathcal E:=k\langle V\rangle$ is a Picard-Vessiot extension of $k$ for $M.$ We first claim that every solution in $L$ of (\ref{alg-indp-Riccati}) belongs to $\mathcal E.$

Let $u\in L$ be a solution of the Riccati equation (\ref{alg-indp-Riccati}). Let  $L^*$ be a Picard-Vessiot extension of $E$ for  the differential equation $y'=-auy.$ Then $L^*=E(x)$ for some nonzero $x$ such that $x'=-aux.$ Note that $V\subset \mathcal E\subseteq L^*$ and that  $V$ is a two dimensional vector space over $C$. Since $C$ is the field of constants of $L^*$ and  $\mathscr L (x)=0,$ we must have $x\in V\subset \mathcal E.$ Consequently, $u=-x'/(ax)\in \mathcal E.$  

Let $\Omega\subset \mathcal E$ be set of all nonalgebraic solutions of Equation (\ref{alg-indp-Riccati}).  Now we shall show that any four distinct elements of $\Omega$ must be $k-$algebraically dependent. That is, $\mathrm{tr.deg}(k(\Omega)|k)\leq 3.$

The differential Galois group $\gal(\mathcal E|k)$ is a closed subgroup of the algebraic group $\mathrm{GL}(V).$ Since dim $\mathrm{GL}(V)=4,$  if $\gal(\mathcal E|k)$ is a proper closed subgroup  then  dim $\gal(\mathcal E|k)\leq 3$ and we obtain $\mathrm{tr.deg}(k(\Omega)|k)\leq$ $\mathrm{tr.deg}(\mathcal E|k)=$ dim $\gal(\mathcal E|k)\leq 3.$   Thus we only need to deal with the case when $\gal(\mathcal E| k)=\mathrm{GL}(V).$ In this case, we first choose a $C-$basis $\{y_1,y_2\}$ of $V$ and identify $\gal(\mathcal E|k)$ with $\mathrm{GL}(2,C).$ Next, we consider the differential field $K=\mathcal E^\mathcal Z,$  where $\mathcal Z$ is  the center  of $\mathrm{GL}(2,C).$ Then for any nonzero solution $v=c_1y_1+c_2y_2\in V$ and any automorphism  $\tau\in \mathcal Z,$ we have $\tau (v)=c_1c_\tau y_1+c_2 c_\tau y_2.$ Thus $\tau (v)=c_\tau v$  and we obtain $v'/v\in K.$ This implies  $\Omega\subset K.$ Since $\mathcal Z$ is normal, $K$ is a Picard-Vessiot extension of $k$ with Galois group $$\gal(K|k)\cong \gal(\mathcal E|k)/ \mathcal Z\cong\  \mathrm{PGL}(2,C).$$
 Now $\mathrm{tr.deg}(K|k)=$ dim $\mathrm{PGL}(2,C)=3$ implies  any four distinct elements in $\Omega$ must be algebraically dependent over $k$. We would like to remark that the Riccati equation $y'=-y^2+x$ over the differential field $\C(x)$ with the derivation $d/dx$ has exactly three $\C(x)-$algebraically independent solutions in any Picard-Vessiot extension of $\C(x)$ (see \cite[Example 4.29]{Mag94} or \cite{Nag-20}). 

Suppose that $L$ has  a nonalgebraic solution $u$ satisfying a Weierstrass differential equation;  $u'^2=\alpha^2(4u^3-g_2u-g_3).$ Then since $L$ has $C$ as its field of constants,  from \cite[Lemma 2]{Kol53},  any other nonalgebraic solution of the Weierstrass equation must belong to the field $k(u,u').$ In fact, the points $(1:z:z'/\alpha),$ $(1:y:y'/\alpha)$ of the elliptic curve (\ref{weierstrassianform}) has the following relation; $(1:z:z'/\alpha)=(1:y:y'/\alpha)(1:c_1:c_2),$ where $(1:c_1:c_2)$ is a $C-$point of the curve. With all these facts, we shall now move on to prove the following

\begin{theorem}\label{algebraicdependence-generaltype} Let $L$ be a differential field extension of $k$ having $C$ as its field of constants. 
			
		\begin{enumerate} [(i)] 
		\item If an autonomous differential equation over $C$ is  not of general type then there is at most one $C-$algebraically independent solution of the equation in $L.$\\
		
			\item  If a first order differential equation over $k$ is not of general type then there are at most three $k-$algebraically independent solutions of the equation in $L.$  
		\end{enumerate}

	\end{theorem}

\begin{proof}
	
	Let $f(y,y')=0$ be a differential equation of nongeneral type. We only need to consider the case when the differential equation has a  nonalgebraic solution $u$ in an iterated strongly normal extension of $k.$ Let $u_1, u_2, u_3, u_4$ be  distinct nonalgebraic solutions from $L.$  Replace $k,$ if necessary, by a finite algebraic extension of $k$ so that $k(u,u')$ contains a nonalgebraic solution of a Riccati or a Weierstrass differential equation over $k.$ Now for $i=1,2,3,4,$ consider the natural  differential embeddings $$\psi_i: k(u, u')\to L; \ \ \psi_i(u)= u_i,  \psi_i(u')= u'_i.$$  If  there is a nonalgebraic solution $t\in k(u,u')\setminus k$ of a Riccati  differential equation over $k,$ say $t'=a_2t^2+a_1t+a_0,$ then  $\psi_i(t)'=a_2\psi_i(t)^2+a_1\psi_i(t)+a_0$ for each $i=1,2,3,4.$ Then as noted earlier, $\mathrm{tr.deg}\left(k(\psi_1(t),\dots,\psi_4(t))|k\right)\leq 3.$ Since $\psi_i(t)\in k(u_i),$ each $u_i$ is  algebraic over $k(\psi_i(t))\subseteq k(\psi_1(t),\dots,\psi_4(t)).$ This implies $u_1,\dots, u_4$ are $k-$algebraically dependent.
	
	Similarly, if $t\in k(u,u')\setminus k$ is a nonalgebraic solution of a Weierstrass differential equation then we have  $t'^2=\alpha^2(4t^3-g_2t^2-g_3)$ and for $i=1,2,$ we have  $(\psi_i(t))'^2=\alpha^2(4(\psi_i(t))^3-g_2\psi_i(t)-g_3).$  As noted earlier, we then have $$(1:\psi_1(t):\psi_1(t)'/\alpha)=(1:\psi_2(t):\psi_2(t)'/\alpha)(1:c_1:c_2), \ \ \text{where} \ c_1,c_2\in C.$$
	That is, 	$\psi_1(k(t,t'))=\psi_2(k(t,t'))$ and we obtain that $u_1$ and $u_2$ are algebraically dependent over $k.$ Thus, in this case,  any two nonalgebraic solutions of $f(y,y')=0$ are $k-$algebraically dependent.

	Now we shall consider the special case when  $f(y,y')=0$ is an autonomous differential equation  having a nonalgebraic solution $u$ in an iterated strongly normal extension of $C.$ In this case, we know from Corollary \ref{SNEoverC} that there is an element $t\in C(u,u')$ such that $t'=1$ or $t'=ct$ for some nonzero constant $c.$ Then, $\psi_i(t)'=1$ or $\psi_i(t)'=c\psi_i(t).$ Thus $\psi_1(t)=\psi_2(t)+e$ for some $e\in C$ or $\psi_1(t)=e\psi_2(t)$ for some nonzero $e\in C.$ This shows that  $\psi_1(C(t))=\psi_2(C(t))$ and we obtain that $u_1$ and $u_2$ are algebraically dependent over $C.$   	\end{proof}

	\subsection{Rational autonomous differential equations} \label{rationalautoconjecture}  We shall now verify  Conjecture \ref{conjecture}
	 for the class of  rational autonomous differential equations $y'=h(y),$ where   $h$ is a nonzero   rational function over $C.$ Let $f(y,y')=y'-h(y).$ We first make the following observations.
	\begin{enumerate}[(a)] \item The differential field $C(f)=C(y)$ with $y'=h(y),$  has $C$ as its field of  constants. This is easily seen by noting that if $u'=0$ for $u\in C(y)\setminus C$ then $y$ is algebraic over the differential field $C(u),$ which contains only constants. An easy calculation using the monic irreducible polynomial of $y$ over $C(u)$ shows that $y'=0.$  This implies $h=0$, a contradiction.  Thus every rational autonomous differential equation admits a nonalgebraic solution.\\
		
		\item If $h$ has no zero in $C$ then $1/h$ is a polynomial over $C$ and has the form 
		(\ref{auto-antideriv}) described in Remark \ref{rational-autoclass}. Thus if $C(y)$ is a transcendental extension of $C$ with $y'=h(y)$ then $C(y)$ has a nonzero element $z$ such that $z'=1.$  This shows that  rational autonomous equations; $y'=h(y)$ with $h$ having no zeros in $C$, are not of general type.\end{enumerate}
		
Suppose that $\alpha\in C$ is a zero of $h.$ Consider the rational function $g\in C(y)$ defined by $h(y):=g(y-\alpha).$ Then, we have $$(y-\alpha)'=y'=h(y)=g(y-\alpha).$$
Thus, given a nonzero rational autonomous differential equation $y'=h(y)$ with $h$ having a zero in $C,$ it does no harm to assume that $h$ has a zero at $y=0.$ In view of Theorem \ref{algebraicdependence-generaltype}  the conjecture is verified once we prove the following:  If an autonomous differential equation \begin{equation*} y'=h(y),\quad h\neq 0 \quad \text{and} \quad h(0)=0\end{equation*} has at most one nonalgebraic solution in any given differential field extension $L$ of $C$ having $C$ as its field of constants then the equation is not of general type.

Consider the purely transcendental differential field extension $C(t,y)$ of $C,$ where   $t'=h(t)$ and $y'=h(y).$  From our hypothesis, there is an element $ u \in C(t,y) \setminus C(t)$ such that $u'=0$.  Then we have following equations  \begin{equation*}
\notag y'=h(y)=\sum^\infty_{i=m}c_i y^i,\qquad u= \sum^\infty_{i=p}b_i y^i,\end{equation*}
where $m\geq 1,$ $c_i\in C$ for all $i\geq m$ and $c_m\neq 0,$ $p$ is an integer and $b_i\in \overline{C(t)}$ for all $i\geq p$ and $b_p\neq 0.$
Taking derivatives, we obtain 
\begin{equation*}0=u'=\sum^\infty_{i=p}b'_i y^i +\left(\sum^\infty_{i=m}c_i y^i\right)\left(\sum^\infty_{i=p}ib_iy^{i-1}\right).
\end{equation*}

We observe that

\begin{enumerate}[(i)]
		\item\label{p0m2} If $p=0$ and $m\geq 2$ then letting $l$ be the least positive integer such that $b_l\neq 0,$  we get $b_l'=0$ and that $b'_{l+m-1}=-l c_m b_l \in C\setminus \{0\} $. \\
	\item \label{p1m2} If $p\neq 0$ and $m\geq 2$ then $b'_p=0$ and $b'_{p+m-1}=-pb_pc_m\in C\setminus\{0\}.$\\
		\item \label{p0m1} If $p=0$ and $m=1$ then letting $l$ be the least positive integer such that $b_l\neq 0, $ we get $b'_l/b_l =-lc_1\in C\setminus\left\lbrace 0\right\rbrace $. \\
\item \label{p1m1}  If $p\neq 0$ and $m=1$ then $b'_p/b_p=-pc_1\in C\setminus \{0\}.$
 \end{enumerate}

Thus, in the event that  (\ref{p0m2}) or (\ref{p1m2}) holds, we obtain an element $z\in C(t)$ such that $z'=1$ and in the event that (\ref{p0m1}) or (\ref{p1m1}) holds, we obtain an element $z\in C(t)\setminus \{0\}$ such that $z'=cz$ for some $c\in C\setminus \{0\}.$ Now from Remark \ref{rational-autoclass}, the equation $y'=h(y)$ is not of general type. Thus, we have verified the  conjecture for the class of rational autonomous differential equations.


\bibliography{KRS}
\bibliographystyle{abbrv}

\end{document}